\documentclass{amsart}
\usepackage[utf8]{inputenc}

\usepackage{mathtools}
\usepackage{amssymb}
\usepackage{accents}
\usepackage{mathrsfs}
\usepackage{bm}
\usepackage[all]{xy}
\UseComputerModernTips
\CompileMatrices
\usepackage[svgnames]{xcolor}
\usepackage[colorlinks,allcolors=blue]{hyperref}
\usepackage{ifthen}
\usepackage{latexsym}
\usepackage{enumitem}
\setlist[enumerate,1]{label=\textup{(\arabic*)}}
\usepackage{tikz}
\usepackage{tikz-3dplot}
\usetikzlibrary{backgrounds}
\usetikzlibrary{arrows.meta}
\allowdisplaybreaks
\numberwithin{equation}{section}
\newtheorem{theorem}[equation]{Theorem}
\newtheorem{lemma}[equation]{Lemma}
\newtheorem{proposition}[equation]{Proposition}
\newtheorem{corollary}[equation]{Corollary}

\theoremstyle{definition}
\newtheorem{definition}[equation]{Definition}

\theoremstyle{remark}
\newtheorem{remark}[equation]{Remark}

\renewcommand{\phi}{\varphi}
\DeclareMathSymbol{\boxprod}{\mathbin}{AMSa}{"03} 
\DeclareMathSymbol{\mixprod}{\mathbin}{AMSa}{"4F} 

\newcommand{\dirsum}{\oplus}

\newcommand{\disjunion}{\sqcup}

\newcommand{\dual}{^\vee}

\newcommand{\homeo}{\approx}

\newcommand{\includesin}{\hookrightarrow}
\newcommand{\intersect}{\cap}
\newcommand{\iso}{\cong}
\newcommand{\Mackey}[1]{{\underline {#1}}}

\newcommand{\smsh}{\wedge}

\newcommand{\Susp}{\Sigma}
\newcommand{\susp}{\Sigma}
\newcommand{\tensor}{\otimes}
\newcommand{\union}{\cup}

\newcommand{\C}{{\mathbb C}}

\newcommand{\PP}{\mathbb{P}}
\newcommand{\R}{{\mathbb R}}

\newcommand{\Z}{\mathbb{Z}}

\newcommand{\HS}{\Mackey{\mathbb{H}}}
\newcommand{\HH}{\Mackey H_\GG}
\newcommand{\HHR}{{\Mackey{\widetilde H}}\vphantom{H}_\GG}
\newcommand{\HR}{{\widetilde H}\vphantom{H}}

\newcommand{\cwd}[1][]{\widehat{c}_\omega^{\ifthenelse{\equal{#1}{}}{}{{\:#1}}}}
\newcommand{\cxwd}[1][]{\widehat{c}_{\chiw}^{\ifthenelse{\equal{#1}{}}{}{{\:#1}}}}

\newcommand{\cld}[1][]{\widehat{c}_{\lambda}^{\ifthenelse{\equal{#1}{}}{}{{\:#1}}}}
\newcommand{\cxld}[1][]{\widehat{c}_{\chi\lambda}^{\ifthenelse{\equal{#1}{}}{}{{\:#1}}}}

\newcommand{\clod}[1][]{\widehat{c}_{\omega_1}^{\ifthenelse{\equal{#1}{}}{}{{\:#1}}}}
\newcommand{\cxlod}[1][]{\widehat{c}_{\chi\omega_1}^{\ifthenelse{\equal{#1}{}}{}{{\:#1}}}}
\newcommand{\cltd}[1][]{\widehat{c}_{\omega_2}^{\ifthenelse{\equal{#1}{}}{}{{\:#1}}}}
\newcommand{\cxltd}[1][]{\widehat{c}_{\chi\omega_2}^{\ifthenelse{\equal{#1}{}}{}{{\:#1}}}}
\newcommand{\cltensd}[1][]{\widehat{c}_{\omega_1\tensor\omega_2}^{\ifthenelse{\equal{#1}{}}{}{{\:#1}}}}
\newcommand{\cxltensd}[1][]{\widehat{c}_{\chi\omega_1\tensor\omega_2}^{\ifthenelse{\equal{#1}{}}{}{{\:#1}}}}
\newcommand{\cd}[1][]{\widehat{c}^{\ifthenelse{\equal{#1}{}}{}{{\:#1}}}}
\newcommand{\Cpq}[2]{\C^{#1+#2\sigma}}
\newcommand{\Cp}[1]{\C^{#1}}
\newcommand{\Cq}[1]{\C^{#1\sigma}}
\newcommand{\Xpq}[2]{\PP(\Cpq{#1}{#2})}

\newcommand{\Xp}[1]{\PP(\C^{#1})}
\newcommand{\Xq}[1]{\PP(\Cq{#1})}
\newcommand{\chiw}{\chi\omega}

\newcommand{\Qpq}[2]{Q_{#1,#2}}

\newcommand{\eQp}[1]{Q_{#1}}        
\newcommand{\Qexpq}[2]{Q(\Cpq {#1}{#2})}   
\newcommand{\Qexp}[1]{Q(\Cp {#1})}
\newcommand{\Qexq}[1]{Q(\Cq {#1})}

\renewcommand{\vec}[1]{\accentset{\rightharpoonup}{#1}} 
\newcommand{\Spqk}[3]{\mathbb{X}_{{#1},{#2}}^{#3}}                  
\newcommand{\gr}{\Diamond}      
\newcommand{\divp}[1][]{\phi_{#1}}
\newcommand{\divq}[1][]{\phi^\chi_{#1}}
\newcommand{\sorb}[1]{{\widehat{\mathscr{O}}_{#1}}}

\newcommand{\rels}[1]{\langle #1 \rangle}

\DeclareMathOperator{\Hom}{Hom}

\DeclareMathOperator{\grad}{grad}

\newcommand{\GG}{{C_2}}
\begin{document}

\title[Complex quadrics II]
{The $\GG$-equivariant ordinary cohomology of complex quadrics II: The symmetric case}
\date{\today}

\author{Steven R. Costenoble}
\address{Steven R. Costenoble\\Department of Mathematics\\Hofstra University\\
  Hempstead, NY 11549, USA}
\email{Steven.R.Costenoble@Hofstra.edu}
\author{Thomas Hudson}
\address{Thomas Hudson, College of Transdisciplinary Studies, DGIST, 
Daegu, 42988, Republic of Korea}
\email{hudson@dgist.ac.kr}



\abstract
In this, the second of three papers about $\GG$-equivariant complex quadrics,
we calculate the equivariant ordinary cohomology of smooth symmetric quadrics
graded on the representation ring of $\Pi BU(1)$
and with coefficients in the Burnside Mackey functor.
These calculations exhibit various interesting properties, including the first
naturally occurring
example we are aware of where the cohomology is not just the sum of shifted copies of
the cohomology of a point, but also has summands that are shifted copies
of the cohomology of the free orbit $\GG/e$.
\endabstract

\maketitle
\tableofcontents

\section{Introduction}

(Note: This is a preliminary version and will be rewritten so that its nomenclature and notation
is consistent with that of \cite{CH:QuadricsI}.)

This is the second of three papers about the $\GG$-equivariant ordinary cohomology
of smooth equivariant projective quadric varieties over $\C$.
The first, \cite{CH:QuadricsI},  which covers antisymmetric quadrics, 
has an appendix with background material and notations that
will assume here.
In this paper we discuss
symmetric quadrics, ones defined by quadratic polynomials $f$
with $tf = f$ where $t$ is the generator of $\GG$.
It is convenient to divide them into four types.
In what follows, we write $\Cpq pq$ for the sum of $p$ copies of
the trivial complex representation $\C$ and $q$ copies of the nontrivial
representation $\Cq{}$.

\newpage
\begin{definition}\ 
\begin{itemize}
    \item \emph{Quadrics of type (D,D)} are subvarieties of $\Xpq{2p}{2q}$ of the form
    \begin{multline*}
        \Qpq{2p}{2q} = \Qexpq{2p}{2q} = \\
        \bigl\{ [x_1:\cdots:x_p:u_p:\cdots:u_1:y_1:\cdots:y_q:v_q:\cdots:v_1] \in \Xpq{2p}{2q} \mid \\
                            \textstyle \sum_i x_iu_i + \sum_j y_jv_j = 0 \bigr\}.
    \end{multline*}

    \item \emph{Quadrics of type (B,D)} are subvarieties of $\Xpq{(2p+1)}{2q}$ of the form
    \begin{multline*}
        \Qpq{2p+1}{2q} = \Qexpq{(2p+1)}{2q} = \\
        \bigl\{ [x_1:\cdots:x_p:z:u_p:\cdots:u_1:y_1:\cdots:y_q:v_q:\cdots:v_1] \in \Xpq{(2p+1)}{2q} \mid \\
                            \textstyle \sum_i x_iu_i + z^2 + \sum_j y_jv_j = 0 \bigr\}.
    \end{multline*}

    \item \emph{Quadrics of type (D,B)} are subvarieties of $\Xpq{2p}{(2q+1)}$ of the form
    \begin{multline*}
        \Qpq{2p}{2q+1} = \Qexpq{2p}{(2q+1)} = \\
        \bigl\{ [x_1:\cdots:x_p:u_p:\cdots:u_1:y_1:\cdots:y_q:w:v_q:\cdots:v_1] \in \Xpq{2p}{(2q+1)} \mid \\
                            \textstyle \sum_i x_iu_i + \sum_j y_jv_j + w^2 = 0 \bigr\}.
    \end{multline*}

    \item \emph{Quadrics of type (B,B)} are subvarieties of $\Xpq{(2p+1)}{(2q+1)}$ of the form
    \begin{multline*}
        \Qpq{2p+1}{2q+1} = \Qexpq{(2p+1)}{(2q+1)} = \\
        \bigl\{ [x_1:\cdots:x_p:z:u_p:\cdots:u_1:y_1:\cdots:y_q:w:v_q:\cdots:v_1] \in \Xpq{(2p+1)}{(2q+1)} \mid \\
                            \textstyle \sum_i x_iu_i + z^2 + \sum_j y_jv_j + w^2 = 0 \bigr\}.
    \end{multline*}
\end{itemize}
\end{definition}

We will often use the abbreviated notation
\[
    [\vec x:z:\vec u:\vec y:w:\vec v] = [x_1:\cdots:x_p:z:u_p:\cdots:u_1:y_1:\cdots:y_q:w:v_q:\cdots:v_1]
\]
and the obvious variations for the other types.

Nonequivariantly, quadrics of type (D,D) and (B,B) are of type D, while those of type
(B,D) and (D,B) are of type B. The fixed points are given by
\[
    \Qexpq mn^\GG = \Qexp m \disjunion \Qexq n.
\]
If $m\neq 2$ and $n\neq 2$, then the induced map
$RO(\Pi BU(1)) \to RO(\Pi\Qpq mn)$ is surjective
(bijective if $m, n > 2$).
If $m=1$ or $n=1$, the corresponding component of the fixed set is empty
and our results will be valid with appropriate interpretations of the formulas.
If $m=2$ or $n=2$, $\Qexp 2$ consists of two points and the map on gradings is
no longer surjective.
In this paper, we will grade all equivariant cohomology on $RO(\Pi BU(1))$
nonetheless, and write $\HH^\gr$ to mean $\HH^{RO(\Pi BU(1))}$.
In the third paper in this series, \cite{CH:QuadricsIII}, we will return to the cases
$m=2$ or $n=2$ and compute the cohomology of those quadrics with their natural,
larger, grading.

We recall the well-known nonequivariant result; see, for example, Lemmas~1 and~2 of \cite{EdidinGraham:quadricbundles}.


\begin{proposition}\label{prop:nonequivariant}
The cohomology of a quadric of type B is given by
\[
    H^*(\Qexp{2p+1}) \iso \Z[c,y]/\rels{c^p - 2y,\ y^2}
\]
where $\grad c = 2$ and $\grad y = 2p$.
A basis is given by
\[
    \{ 1, c, c^2, \ldots, c^{p-1}, y, cy, \ldots, c^{p-1}y \},
\]
with one basis element in each even grading from $0$ to $2(2p-1)$.
(If $p = 0$, $\eQp 1 = \emptyset$ and these relations force the cohomology to be $0$.)

If $p > 1$, the cohomology of a quadric of type D is given by
\[
    H^*(\Qexp{2p}) \iso \Z[c,y]/\rels{c^p - 2cy,\ y^2 - \epsilon c^{p-1}y}
\]
where
\[
    \epsilon =
    \begin{cases}
        0 & \text{if $p$ is even} \\
        1 & \text{if $p$ is odd.}
    \end{cases}
\]
Here, $\grad c = 2$ and $\grad y = 2(p-1)$.
A basis is given by
\[
    \{ 1, c, c^2, \ldots, c^{p-1}, y, cy, \ldots, c^{p-1}y \},
\]
with one basis element in each even grading from $0$ to $2(2p-2)$ except for
grading $2(p-1)$, where there are two basis elements, $c^{p-1}$ and $y$.
An alternative basis often used is
\[
    \{ 1, c, c^2, \ldots, c^{p-2}, c^{p-1} - y, y, cy, \ldots, c^{p-1}y \}
\]

In case of the type D quadric $\Qexp 2$, which consists of two points, we can write
\[
    H^*(\Qexp 2) \iso \Z[y]/\rels{y^2 - y}
\]
where $\grad y = 0$. 
We can take as a basis for the only nonzero group either
\[
    \{ 1, y \} \qquad\text{or}\qquad \{ 1-y, y \}.
\]
Both $y$ and $1-y$ are idempotents.
\qed
\end{proposition}

With some abuse of notation, the calculation of $H^*(\Qexp 2)$ can be considered a special case
of the calculation for general type D quadrics. The abuse is that we need to interpret $c^0 = 1$,
while the relations imply that $c = 0$ in this case.

We will calculate the equivariant cohomologies of symmetric quadrics in
Theorems~\ref{thm:bb mutliplicative}, \ref{thm:db multiplicative},
and~\ref{thm:dd multiplicative}.
The results are somewhat different for each type, as might be expected from the nonequivariant
calculations. The outlines of the arguments are similar, though, and each depends on the existence
of certain elements that exhibit a \emph{divisibility} phenomenon we first saw
in calculating the cohomology of finite projective spaces in \cite{CHTFiniteProjSpace}.
There the divisibility had a clear geometric explanation, but here it is a bit more subtle,
and depends on a general result we prove as Proposition~\ref{prop:general divisibility},
which we expect will see more use in future calculations.

The result for quadrics of type (B,B) has a feature anticipated by Ferland and Lewis
in \cite{FerlandLewis} that we have not yet seen in previous calculations.
They showed that, under conditions satisfied by, for example, smooth complex varieties,
equivariant cohomology is free over the cohomology of a point. However, by this they meant
that the cohomology could be the sum of copies of
$\HH^{RO(\GG)}(\GG/\GG)$ \emph{and also} copies of $\HH^{RO(\GG)}(\GG/e)$.
They called generators of summands of the form $\HH^{RO(\GG)}(\GG/\GG)$
generators \emph{of type $\GG/\GG$} and generators of summands of the form
$\HH^{RO(\GG)}(\GG/e)$ generators \emph{of type $\GG/e$}.
(See the following section for more details.)
Our result for quadrics of type (B,B) can be stated as follows,
and will be proved as Theorem~\ref{thm:bb mutliplicative}.

\begin{theorem}
As an algebra over $\HH^\gr(BU(1))$, $\HH^\gr(\Qpq{2p+1}{2q+1})$ is generated by
an element $x_{p,q}$ of type $\GG/\GG$ and grading $(p+1)\omega + (q+1)\chi\omega - 2$
and an element $y$ of type $\GG/e$ and grading $2(p+q)$,
subject to the facts that
\begin{align*}
    \divp &= \cwd[p] - e^{-2(q+1)}\kappa\zeta_1^q x_{p,q} \\
\intertext{is infinitely divisible by $\zeta_0$ and}
    \divq &= \cxwd[q] - e^{-2(p+1)}\kappa\zeta_0^p x_{p,q}
\end{align*}
is infinitely divisible by $\zeta_1$,
and the following relations, where $c = \zeta^{-1}\rho(\cwd)$:
\begin{align*}
    c^{p+q+1} &= 2cy  \\
    y ^2 &= \begin{cases}
                    0 &\text{if $p+q+1$ is even} \\
                    c^{p+q}y  &\text{if $p+q+1$ is odd}
                 \end{cases} \\
    \rho(x_{p,q}) &= \iota^{2(q+1)}\zeta^{p-q} cy  \\
    x_{p,q}^2 &= 0 \\
    \divp \divq &= \tau(\iota^{2q}\zeta^{p-q}y ).
\end{align*}
\qed
\end{theorem}
The first two relations say that the nonequivariant cohomology is that
of the underlying quadric of type D while the others give the truly equivariant relations.
As a module over the cohomology of a point, 
for each coset $m\Omega_1 + RO(\GG)\subset RO(\Pi BU(1))$,
the cohomology groups $\HH^{m\Omega_1+RO(\GG)}(\Qpq{2p+1}{2q+1})$
have $2(p+q)$ generators of type $\GG/\GG$ and one of type $\GG/e$,
giving the first natural example that we are aware of with a generator of that last type.

The cohomologies of symmetric quadrics $\Qpq mn$ of type (B,D), (D,B), and (D,D) have generators only of type $\GG/\GG$,
so are simpler in that respect. However, they have other interesting features due to the interplay
between the parities of $m$, $n$, and $m+n$.

We offer these calculations as additions to the stable of known calculations
in equivariant ordinary cohomology
and as examples of the benefits of using the extended grading.
In \cite{CH:QuadricsI} we used the calculation for one of the antisymmetric quadrics to give an
equivariant refinement of the result that there are 27 lines on a general cubic surface
in $\PP^3$.
In \cite{CH:QuadricsIII}, once we have expanded further the grading on the exceptional symmetric cases,
we will use one of those cases to give another refinement of the 27 lines result.

\section{Mackey functor-valued ordinary cohomology}\label{sec:mackey}

We mentioned in the introduction that the cohomology of a quadric of type (B,B) has
summands isomorphic to $\HH^{RO(\GG)}(\GG/e)$ as well as ones isomorphic to $\HH^{RO(\GG)}(\GG/\GG)$.
Because of this, it is more natural in this context to think of ordinary cohomology as
taking values in (graded) Mackey functors rather than groups.
That is indicated by the underline in the notation. 
We will also use $\HH$ to denote unreduced cohomology, and
use $\HHR$ for reduced cohomology when needed.

A Mackey functor is a functor on the stable orbit category $\sorb\GG$,
and we picture such a functor $\Mackey T$ as
\[
 \xymatrix{
  \Mackey T(\GG/\GG) \ar@/_/[d]_{\rho} \\
  \Mackey T(\GG/e) \ar@/_/[u]_{\tau} \ar@(dl,dr)[]_{t}
 }
\]
where $\rho$ is restriction, $\tau$ is the transfer, and $t$ is the action of the nontrivial
element $t\in\GG$.
For any $\Mackey T$ we have that $\tau\rho$ is multiplication by $g = [\GG/e] \in A(\GG)$,
which gives the action of $A(\GG)$ on $\Mackey T(\GG/\GG)$.
The map $t$ defines an action of $\Z[\GG]$ on $\Mackey T(\GG/e)$, and
$\rho\tau$ is multiplication by $1+t$.

If $\Mackey T$ is a Mackey functor, we should be careful when discussing elements
to write either $x\in \Mackey T(\GG/\GG)$ or $x\in \Mackey T(\GG/e)$.
However, we will so often be interested in the former that we will write
$x\in\Mackey T$ to mean $x\in\Mackey T(\GG/\GG)$.

There are two commonly used Mackey functors, the projective functors
$\Mackey A = \sorb\GG(-,\GG/\GG)$ and $\Mackey A_{\GG/e} = \sorb\GG(-,\GG/e)$,
the first of which is the Burnside ring Mackey functor which we take
as the coefficients for ordinary cohomology throughout.
We can picture these as follows.
\[
 \Mackey A\colon \quad
     \xymatrix{
		A(\GG) \ar@/_/[d]_{\epsilon} \\
		\Z \ar@/_/[u]_{\cdot g} \ar@(dl,dr)[]_{1}
	}
\qquad\qquad
 \Mackey A_{\GG/e}\colon 
    \xymatrix{
		\Z \ar@/_/[d]_{\cdot (1+t)} \\
		\Z[\GG] \ar@/_/[u]_{\epsilon} \ar@(dl,dr)[]_{\cdot t}
	  }
\]
Here, $\epsilon\colon A(\GG)\to\Z$ is the cardinality map $\epsilon(a+bg) = a + 2b$
and $\epsilon\colon \Z[\GG]\to \Z$ is $\epsilon(a+bt) = a + b$.
We think of $\Mackey A$ as generated by the element $1\in \Mackey A(\GG/\GG) = A(\GG)$
and $\Mackey A_{\GG/e}$ as generated by $1\in \Mackey A_{\GG/e}(\GG/e) = \Z[\GG]$.
Related to this is that
\begin{align*}
    \Hom_{\sorb\GG}(\Mackey A,\Mackey T) &\iso \Mackey T(\GG/\GG) \\
\intertext{and}
    \Hom_{\sorb\GG}(\Mackey A_{\GG/e},\Mackey T) &\iso \Mackey T(\GG/e),
\end{align*}
in each case via the image of $1$.

To abbreviate the notation, we will write
\begin{align*}
    \HS &= \HH^{RO(\GG)}(\GG/\GG) \\
\intertext{and}
    \HS_{\GG/e} &= \HH^{RO(\GG)}(\GG/e).
\end{align*}
Because ordinary cohomology is ``ordinary'' and we use $\Mackey A$ as the coefficient system, it follows that
\begin{align*}
    \HS^0 &\iso \Mackey A \\
\intertext{and}
    \HS_{\GG/e}^0 &\iso \Mackey A_{\GG/e}.
\end{align*}
$\HS$ is a graded Mackey functor-valued ring, which we can take as meaning that
both $\HS(\GG/\GG)$ and $\HS(\GG/e)$ are graded rings, $\rho$ is a ring map,
and $\tau$ is a map of $\HS(\GG/\GG)$-modules.
The structure of $\HS(\GG/\GG)$ is discussed in more detail in the Appendix to
\cite{CH:QuadricsI} and we will use the notations introduced there for its elements.
The ring $\HS(\GG/e)$ is quite simple, being the nonequivariant
cohomology of a point regraded on $RO(\GG)$, which gives
$\HS(\GG/e) \iso \Z[\iota^{\pm 1}]$ where $\iota \in \HS^{\sigma-1}(\GG/e)$.

We have $\HS_{\GG/e}(\GG/\GG) \iso \HS(\GG/e)$ and that
$\HS_{\GG/e}$ is a projective $\HS$-module. In fact,
\[
    \HS_{\GG/e} \iso \Mackey A_{\GG/e}[\iota^{\pm 1}]
\]
where $\iota\in \HS_{\GG/e}^{\sigma-1}(\GG/\GG)$.
$\HS_{\GG/e}$ is generated as an $\HS$-module by $1\in \HS_{\GG/e}^0(\GG/e) = \Z[\GG]$.
Note that we are thinking of this as a module generator;
as an abelian \emph{group}, $\HS_{\GG/e}^0(\GG/e)$ is generated by two elements, $1$ and $t$.
This distinction will show up when we calculate the cohomology of quadrics of type (B,B);
see Remark~\ref{rem:bb divp times divq}.

As mentioned in the introduction, and following Ferland and Lewis \cite{FerlandLewis},
by a free $\HS$-module we mean the following.

\begin{definition}
A Mackey-functor valued graded module over $\HS$ is said to be \emph{free}
if it is a direct sum of shifted copies of $\HS$ and $\HS_{\GG/e}$.
A generator of a summand of the form $\susp^\alpha \HS$ will be called
a basis element \emph{of type $\GG/\GG$} and a generator of a summand
of the form $\susp^\alpha\HS_{\GG/e}$ will be called
a basis element \emph{of type $\GG/e$}.
\end{definition}

Explicitly, if $\Mackey F$ is such a free module, a generator of type $\GG/\GG$ lives
in $\Mackey F(\GG/\GG)$ while a generator of type $\GG/e$ lives in $F(\GG/e)$.

Because $\HS$ has only one grading that looks like $\Mackey A$, the gradings
of generators of type $\GG/\GG$ in a free $\HS$-module are well-defined;
Ferland and Lewis mention a possible ambiguity for ambient cyclic groups $C_p$,
but that occurs only when $p > 3$.
However, because of the invertible element $\iota$, the gradings of generators
of type $\GG/e$ are determined only up to multiples of $\sigma - 1\in RO(\GG)$.

Ferland and Lewis showed that $\GG$-cell complexes with even-dimensional cells
will have cohomology free over $\HS$ in the sense above,
with summands of $\HS_{\GG/e}$ necessary if we allow free cells.
(A similar freeness result was shown for more general cell complexes
in \cite{KronholmFree}, as corrected in \cite{HogleMayFreeness},
when the coefficient system is the constant $\Mackey\Z/2$ Mackey functor.)
Previous computations we are aware of have resulted in cohomologies
that are direct sums of copies of $\HS$, and one of the interesting
aspects of our computations here is that the cohomology of quadrics of type (B,B)
has a summand of type $\HS_{\GG/e}$.

\section{Divisibility}

This section is a bit of a digression to establish a technical result we will need in
the calculation. If the reader is willling to accept Proposition~\ref{prop:general divisibility}
below, this section can be skipped over on first reading.
But we expect this result to be helpful in other calculations in the future.

One of the phenomena that we saw occur with finite projective spaces
in \cite{CHTFiniteProjSpace} and antisymmetric quadrics in \cite{CH:QuadricsI}
is the presence of elements that are infinitely divisible
by $\zeta_0$ or $\zeta_1$. We will see the same phenomenon with symmetric quadrics,
but the existence of some of these elements is a little more delicate and
depends on the following general result.

Recall from \cite[Part~3]{Co:InfinitePublished} that there is a $\GG$-space $K(2)$,
characterized by the fact that it is nonequivariantly contractible and its fixed set
$K(2)^\GG = K^0\disjunion K^1$ consists of two contractible components. 
The representation ring of its fundamental groupoid is the same as that of $BU(1)$
and we will use the same names for its elements.

This space can be thought of as
classifying the component structure of the fixed sets of $\GG$-spaces as follows:
Suppose that $X$ is a $\GG$-space with its fixed set $X^\GG = X^0\disjunion X^1$ written
as the disjoint union of two (not necessarily connected) subspaces. Then there is, up to $\GG$-homotopy,
a unique $\GG$-map $X\to K(2)$ taking $X^0$ to $K^0$ and $X^1$ to $K^1$.
For example, there is such a map $BU(1)\to K(2)$ with the usual splitting
of $BU(1)^\GG = \Xp\infty \disjunion \Xq\infty$.

Moreover, there are elements $\zeta_0\in \HHR^{\Omega_0}(K(2)_+)$ and
$\zeta_1\in \HHR^{\Omega_1}(K(2)_+)$ mapping to the elements of the same names in the cohomology
of $BU(1)$. For a general map $X\to K(2)$ as above, we again use the names $\zeta_0$ and $\zeta_1$
for the images of these elements in the cohomology of $X$, and we call $\zeta_0$ the cohomology
element corresponding to $X^0$ and $\zeta_1$ the element corresponding to $X^1$.
These elements will always satisfy $\zeta_0\zeta_1 = \xi$, and will have the properties that
$\zeta_0$ restricted to $X^1$ is invertible, as is $\zeta_1$ restricted to $X^0$.

We next recall a lemma from \cite{Co:InfinitePublished}, which was slightly misstated there.
The proper statement, and what was actually proved, is the following,
where $X$ is again a $\GG$-space with $X^\GG = X^0\disjunion X^1$ as above.

\begin{lemma}[Lemma~16.2 of \cite{Co:InfinitePublished}]\label{lem:first divisibility}
Let $f\colon Y\to X$ be an ex-$\GG$-space
and let $Y^0 = (f^\GG)^{-1}(X^0) \subseteq Y^\GG$.
If $Y^0$ is contractible to the base section over $X$, then $\zeta_0$ acts by isomorphisms
on $\HHR^\gr(Y)$.
\qed
\end{lemma}

With this we can prove the following generalization of \cite[Proposition~16.1]{Co:InfinitePublished}.
We let $i_0 \colon X^0\to X$ be the inclusion and let $\eta_0 = i_0^*$.

\begin{proposition}\label{prop:general divisibility}
If $x\in \HHR^\alpha(X_+)$ and there exists a $y\in \HR^\alpha(X^0_+)$ such that $\eta_0(x) = \tau(y)$,
then $x$ is infinitely divisible by $\zeta_0$.
\end{proposition}

\begin{proof}
Embed $\GG = \{1,-1\}\subset \R^\sigma$ and consider the collapse map
\[
    c\colon S^\sigma \to \Susp^\sigma (\GG)_+.
\]
The transfer $\tau$ is, up to suspension, induced by the following map over $X$:
\[
    X^0\times S^\sigma \xrightarrow{1\times c} X^0\times \Susp^\sigma(\GG)_+.
\]
Let $Z$ be the homotopy pushout of ex-spaces over $X$ in the following diagram:
\[
    \xymatrix{
        X^0\times S^\sigma \ar[r]^-{c} \ar[d]_{i_0} & X^0\times \Susp^\sigma(\GG)_+ \ar[d] \\
        X\times S^\sigma \ar[r] & Z.
    }
\]
We then have a Mayer-Vietoris long exact sequence
\[
    \cdots \to \HHR^{\alpha+\sigma}(Z) \to
    \HH^{\alpha}(X_+)\dirsum \HHR^{\alpha}(X^0_+\smsh \GG_+)
    \xrightarrow{\eta_0 - \tau}
    \HHR^{\alpha}(X^0_+) \to \cdots
\]
By assumption, $\eta_0(x) - \tau(y) = 0$, so there exists a $z\in \HHR^{\alpha+\sigma}(Z)$ mapping
to $(x,y)$; in particular, $z$ maps to $x$.

From the construction of $Z$, the set $Z^0 \subseteq Z^\GG$ over $X^0$ is contractible to the base section,
hence Lemma~\ref{lem:first divisibility}
implies that every element in $\HHR^\gr(Z)$ is infinitely divisible by $\zeta_0$. 
Since $x$ is the image of $z$, it follows that
$x$ is infinitely divisible by $\zeta_0$.
\end{proof}

Of course, by renaming, this gives a similar result with $\eta_0$ replaced by $\eta_1$, restriction to $X^1$,
and $\zeta_0$ replaced by $\zeta_1$.

\section{The excess intersection formula}\label{sec:excess}

There is another useful but technical result we need,
similar to the excess intersection formula given by Fulton as \cite[Theorem~6.3]{FultonIntersection}.

Suppose that $M$ is a closed $\GG$-manifold with closed $\GG$-submanifolds $N$ and $Z$,
with $Z$ representing a cohomology element in $\HH^\gr(M)$ (so $\dim M - \dim Z \in RO(\Pi M)$).
We assume that $Z' = N\intersect Z$ is a closed manifold, but not that $N$ and $Z$ intersect transversely.
Instead, suppose that there is a bundle $\epsilon$ over $N$, called the
\emph{excess normal bundle,} such that
\[
    \nu(Z', M) = \nu(Z',N) \dirsum \nu(Z',Z) \dirsum (\epsilon|Z').
\]
If the intersection were transverse, we would have $\epsilon = 0$.
Let $i\colon N\to M$ be the inclusion.

\begin{proposition}[Excess Intersection Formula]\label{prop:excess formula}
In the context given just above,
\[
    i^*[Z]^* = e(\epsilon)[Z']^* \in \HH^\gr(N)
\]
where $e(\epsilon) \in \HH^\gr(N)$ is the Euler class of $\epsilon$.
\end{proposition}

\begin{proof}
The class $[Z]^*$ is represented by the map over $M$ given by
\[
    M_+ \xrightarrow{c} T\nu(Z,M) \xrightarrow{t} H^{\nu(Z,M)}
\]
where $c$ is the collapse map, $t$ is the Thom class, and
$H^\gamma$ is the parametrized spectrum over $M$ representing cohomology in gradings $\gamma + RO(\GG)$.
We get $i^*[Z]^*$ by precomposing with $i$.

We then have the following diagram, commutative up to homotopy:
\[
    \xymatrix{
        N_+ \ar[r]^i \ar[d]_c & M_+ \ar[dd]^c \\
        T\nu(Z',N) \ar[d]_{\dirsum 0} \\
        T(\nu(Z',N)\dirsum\epsilon) \ar[r] & T\nu(Z,M)
    }
\]
where the map on the bottom is induced by the isomorphism
\[
    \nu(Z,M) | Z' \iso \nu(Z',N)\dirsum\epsilon
\]
that follows from our assumptions.
The proposition follows from the fact that $e(\epsilon)$ is represented by the composite
$N_+\to T\epsilon\to H^\epsilon$ where the first map is the zero section.
\end{proof}

Our main application of this formula is Lemma~\ref{lem:bb restriction}.

\section{Quadrics of type (B,B)}

We look first at quadrics of type (B,B) because they will serve as a basis for
calculating the cohomologies of the other types.
Our calculation in each case will use a cofibration whose cofiber
is a suspension of $\Xpq pq$, which we include in $\Qexpq mn$ 
(where $m = 2p$ or $2p+1$ and $n = 2q$ or $2q+1$)
via the map
\begin{equation}\label{def:j}
    j\colon \Xpq pq \to \Qexpq mn, \qquad j[\vec u:\vec v] = [\vec 0_{m-p}:\vec u:\vec 0_{n-q}:\vec v]
\end{equation}
(where $\vec 0_k$ indicates $k$ zeros), which has normal bundle
\[
    \nu \iso (m-p)\omega \dirsum (n-q)\chi\omega - O(2).
\]

In the (B,B) case, our cofibration is the inclusion of
the following subspace that we called
a \emph{binate Schubert variety} in
\cite{CH:geometric}.

\begin{definition}
Let $\Spqk pq{p+q+1} \subset \Qpq{2p+1}{2q+1}$ be the subspace defined by
\[
    \Spqk pq{p+q+1} = \{ [\vec x:z:\vec 0:\vec y:w:\vec 0] \in \Qpq{2p+1}{2q+1} \}.
\]
\end{definition}

Note that the defining equation for $\Qpq{2p+1}{2q+1}$ becomes $z^2 + w^2 = 0$ on
$\Spqk pq{p+q+1}$, hence $w = \pm iz$. So we can write
\[
    \Spqk pq{p+q+1} = P \union tP
\]
where
\[
    P = \{ [\vec x:z:\vec 0:\vec y:iz:\vec 0] \in \Qpq{2p+1}{2q+1} \} \homeo \Xp{p+q+1}
\]
is a nonequivariant projective space, and
\[
    P\intersect tP = \Xpq pq = \{ [\vec x:0:\vec 0:\vec y:0:\vec 0] \in \Qpq{2p+1}{2q+1} \}.
\]

Let $i\colon \Spqk pq{p+q+1}\to \Qpq{2p+1}{2q+1}$ be the inclusion and recall $j$ from (\ref{def:j}).

\begin{proposition}\label{prop:bb cofibration}
We have a cofibration
\[
    (\Spqk pq{p+q+1})_+ \xrightarrow{i} (\Qpq{2p+1}{2q+1})_+ \to \susp^\nu j(\Xpq pq)_+
\]
where $\nu \iso (p+1)\omega\dual\dirsum(q+1)\chi\omega\dual - O(2)$ is the normal bundle to $j$.
\end{proposition}

\begin{proof}
We have a map
\[
    \Qpq{2p+1}{2q+1} \setminus \Spqk pq{p+q+1} \to j(\Xpq pq)
\]
given by
\[
    [\vec x:z:\vec u:\vec y:w:\vec v] \mapsto [\vec 0:0:\vec u:\vec 0:0:\vec v],
\]
which we can identify with the normal bundle to $j$.
\end{proof}

This gives us a long exact sequence in cohomology. In order to use this
sequence for calculations, we need to know the cohomology of $\Spqk pq{p+q+1}$.

\begin{proposition}\label{prop:binate cohomology}
There is a split short exact sequence
\begin{multline*}
    0 \to \HH^\gr(\Spqk pq{p+q+1})
    \to \HH^\gr(\Xpq pq)\dirsum\HH^\gr(\GG\times\Xp{p+q+1}) \\
    \to \HH^\gr(\GG\times\Xpq pq)
    \to 0
\end{multline*}
where the last map is the difference of the two restriction maps.
As a result,
\[
    \HH^\gr(\Spqk pq{p+q+1}) \iso \HH^\gr(\Xpq pq) \dirsum \susp^{2(p+q)}\HS_{\GG/e}[\zeta_0^{-1}]
\]
as a module over $\HS$. 
Multiplicatively, if we use the usual names for elements of $\HH^\gr(\Xpq pq)$
and write $y \in \HH^{2(p+q)}(\Spqk pq{p+q+1})(\GG/e)$ for a generator of the
summand $\susp^{2(p+q)}\HS_{\GG/e}$, the multiplication is determined by
the relations in the cohomology of $\Xpq pq$, except that we replace the relation
$\cwd[p]\cxwd[q] = 0$ with
\[
    \cwd[p]\cxwd[q] = \zeta_0^q\zeta_1^p \tau(y ).
\]
\end{proposition}

\begin{proof}
We have the following pushout diagram:
\[
    \xymatrix{
        \GG\times\Xpq pq \ar[r]\ar[d] & \GG\times\Xp{p+q+1} \ar[d] \\
        \Xpq pq \ar[r] & \Spqk pq{p+q+1}.
    }
\]
Since the top map is a $\GG$-cofibration, we get a Mayer-Vietoris exact sequence
\begin{multline*}
    \cdots \to \HH^\gr(\Spqk pq{p+q+1})
    \to \HH^\gr(\Xpq pq)\dirsum\HH^\gr(\GG\times\Xp{p+q+1}) \\
    \to \HH^\gr(\GG\times\Xpq pq)
    \to \cdots.
\end{multline*}
The map to $\HH^\gr(\GG\times\Xpq pq)$ is clearly surjective, and the target
is a projective $\HS$-module, so we get the split short exact sequence claimed.
This gives the additive structure
(using the fact that $\xi$, hence $\zeta_0$, acts via isomorphisms on $\HH^\gr(\GG\times Y)$
for any nonequivariant space $Y$ over $BU(1)$).

For the multiplicative structure, note that $\cwd\in \HH^\omega(\Spqk pq{p+q+1})$
maps to the pair $(\cwd,\cd) \in \HH^\omega(\Xpq pq)\dirsum\HH^\omega(\GG\times\Xp{p+q+1})$,
and similarly for $\cxwd$. This gives
\[
    \cwd[p]\cxwd[q] \mapsto (\cwd[p]\cxwd[q], \cd[p+q]) = (0, \cd[p+q]) = \zeta_0^q\zeta_1^p\tau(y )
\]
as claimed.
\end{proof}

\begin{remark}\label{rem:binate representation}
Although $\Spqk pq{p+q+1}$ is not a manifold, the preceding proposition allows
us to represent elements in its cohomology using (singular) manifolds
by embedding the cohomology in $\HH^\gr(\Xpq pq)\dirsum\HH^\gr(\GG\times\Xp{p+q+1})$.
For example, we can write
\begin{align*}
    \cwd &= ( [\Xpq{p-1}q]^*, [\GG\times\Xp{p+q}]^*) \\
    y  &= ([\emptyset]^*, [*]^*) \quad\text{as element at level $\GG/e$} \\
    \tau(y ) &= ( [\emptyset]^*, [\GG/e]^* )
\end{align*}
where $*$ represents any point in $\Xp{p+q+1}$ and $\GG/e$ denotes its orbit in $\GG\times\Xp{p+q+1}$.
If we identify $\Xp{p+q+1}$ with the subspace $P\subset \Spqk pq{p+q+1}$ used earlier,
a good choice is $* = [\vec 0:1:\vec 0:\vec 0:i:\vec 0]$ and its orbit 
$\GG/e = \{ [\vec 0:1:\vec 0:\vec 0:\pm i:\vec 0] \}$.
\end{remark}

We now turn to the cohomology of the quadric itself. We give names to the following elements.

\begin{definition}
Let 
\[
    x_{p,q} = j_!(1) = [j(\Xpq pq)]^* \in \HH^{(p+1)\omega + (q+1)\chi\omega - 2}(\Qpq{2p+1}{2q+1}).
\]
where $j$ is as in (\ref{def:j}).
Let
\[
    y  = [P']^* \in H^{2(p+q)}(\Qpq{2p+1}{2q+1}) = \HH^{2(p+q)}(\Qpq{2p+1}{2q+1})(\GG/e)
\]
be the nonequivariant cohomology element dual to
\[
    P' \homeo \Xp{p+q+1} = \{ [\vec 0:z:\vec u:\vec 0:iz:\vec v] \in \Qpq{2p+1}{2q+1} \}
\]
\end{definition}

We will shortly need the following calculation. Let
\[
    \eta\colon \HH^\gr(\Qpq{2p+1}{2q+1}) \to \HH^\gr(\Qpq{2p+1}{2q+1}^\GG)
        \iso \HH^\gr(\Qexp{2p+1})\dirsum\HH^\gr(\Qexq{(2q+1)})
\]
be the map induced by the inclusion of the fixed points.
We call $\eta$ \emph{restriction to fixed points,} which should be distinguished
from the fixed-point map $(-)^\GG$.
Notice that $\zeta_1$ is invertible in $\HH^\gr(\Qexp{2p+1})$ while
$\zeta_0$ is invertible in $\HH^\gr(\Qexq{(2q+1)})$,
because each sits over just one of the components of $\Qpq{2p+1}{2q+1}^\GG$.
Because these spaces have trivial $\GG$-action and free nonequivariant cohomology, we have
\begin{align*}
    \HH^\gr(\Qexp{2p+1}) &\iso \HS[\zeta_0^{\pm 1},c,y]/\rels{c^p - 2y, y^2} && \grad c = 2,\ \grad y = 2p\\
\intertext{and}
    \HH^\gr(\Qexq{(2q+1)}) &\iso \HS[\zeta_1^{\pm 1},c,y]/\rels{c^q - 2y, y^2} && \grad c = 2,\ \grad y = 2q.
\end{align*}

\begin{lemma}\label{lem:bb restriction}
$\eta(x_{p,q}) = ( (e^2 + \xi c)^{q+1}\zeta_1^{p-q}y, (e^2 + \xi c)^{p+1}\zeta_0^{q-p}y )$.
\end{lemma}

\begin{proof}
The first component $\eta_0(x_{p,q})$ is the image of $x_{p,q} = [\Xpq pq]^*$ under the restriction map induced by
$\Qexp{2p+1} \includesin \Qexpq{2p+1}{(2q+1)}$.
Using the terminology of \S\ref{sec:excess}, the intersection
$\Xpq pq \intersect \Qexp{2p+1} = \Xp p$ is not transverse, but has excess normal bundle
$(q+1)\chi\omega\dual$. From Proposition~\ref{prop:excess formula}, we then get
\[
    \eta_0(x_{p,q}) = e((q+1)\chi\omega\dual)[\Xp p]^* = (e^2 + \xi c)^{q+1}\zeta_1^{-(q+1)}\cdot\zeta_1^{p+1} y,
\]
where $e(\chi\omega\dual) = (e^2 + \xi c)\zeta_1^{-1}$ was shown in \cite[Proposition~7.5]{Co:InfinitePublished},
and $[\Xp p]^* = \zeta_1^{p+1}y$ requires the invertible factor $\zeta_1^{p+1}$ to put it in the correct grading.
The calculation of the second component of $\eta(x_{p,q})$ is similar.
\end{proof}

In order to show that the long exact sequence implied by Proposition~\ref{prop:bb cofibration}
splits into short exact sequences, we need to introduce two more elements.

\begin{definition}
Let
\begin{align*}
    \divp = \divp[2p+1] = \divp[2p+1,2q+1] &= \cwd[p] - e^{-2(q+1)}\kappa\zeta_1^q x_{p,q} \\
\intertext{and}
    \divq = \divq[2q+1] = \divq[2p+1,2q+1] &= \cxwd[q] - e^{-2(p+1)}\kappa\zeta_0^p x_{p,q}.
\end{align*}
(We will have similar elements for the other types of quadrics,
and we use as many subscripts as necessary to make clear in context which element we are referring to.)
\end{definition}

The reason for introducing these elements is the following.

\begin{proposition}\label{prop:divisible elements}
In the cohomology of $\Qpq{2p+1}{2q+1}$, $\divp$
is infinitely divisible by $\zeta_0$
and  $\divq$ is infinitely divisible by $\zeta_1$.
\end{proposition}



\begin{proof}
If $p = 0$, every element in the cohomology of $\Qpq{1}{2q+1}$ is divisible by $\zeta_0$,
because $\Qpq{1}{2q+1}^\GG$ does not intersect $\Xp{1}$.
Similarly, if $q = 0$, every element is divisible by $\zeta_1$.

So assume that $p > 0$ and consider $\divp$.
We verify the assumptions of Proposition~\ref{prop:general divisibility}. 
Using Lemma~\ref{lem:bb restriction}, we have
\begin{align*}
    \eta_0(\cwd) &=  \zeta_1 c \\
    \eta_0(x_{p,q}) &= (e^2 + \xi c)^{q+1} \zeta_1^{p-q} y, \\
\intertext{so}
    \eta_0(\cwd[p]) &= \zeta_1^p c^p  = 2\zeta_1^p y \\
    \eta_0(e^{-2(q+1)}\kappa\zeta_1^q x_{p,q}) &= e^{-2(q+1)}\kappa(e^2 + \xi c)^{q+1} \zeta_1^p y \\
        &= \kappa \zeta_1^p y 
\end{align*}
because $e^{-n}\kappa \cdot \xi = 0$. This gives
\[
    \eta_0(\divp) = 2\zeta_1^py - \kappa\zeta_1^p y = g\zeta_1^p y = \tau(\zeta^p y),
\]
hence $\divp$ is infinitely divisible by $\zeta_0$ by Proposition~\ref{prop:general divisibility}.

The argument for $\divq$ is similar.
\end{proof}

We can now prove the following, which gives the additive structore of the cohomology of the quadric.

\begin{proposition}\label{prop:bb splitting}
There is a split short exact sequence
\[
    0 \to
    \susp^{\nu}\HH^\gr(\Xpq pq) \xrightarrow{j_!}
    \HH^\gr(\Qpq{2p+1}{2q+1}) \xrightarrow{i^*}
    \HH^\gr(\Spqk pq{p+q+1}) \to 0
\]
where $\nu = (p+1)\omega + (q+1)\chi\omega - 2$.
\end{proposition}

\begin{proof}
The cofibration of Proposition~\ref{prop:bb cofibration} gives us a long exact sequence,
which we need to show breaks into the short exact sequence claimed.
That this short exact sequence splits then follows from the fact that
$\HH^\gr(\Spqk pq{p+q+1})$ is free over $\HS$.


We need to show that $i^*$ is surjective, so we show that every basis element is hit.
Using Remark~\ref{rem:binate representation} and the excess intersection formula
(in the case where the intersection is transverse, so $\epsilon = 0$), we see that $i^*(y ) = y $,
so what is left are the generators that come from $\Xpq pq\subset \Spqk pq{p+q+1}$.

The inclusions $\Xpq pq\subset \Spqk pq{p+q+1}\subset \Qpq{2p+1}{2q+1}\subset BU(1)$
show that $i^*$ is a map of modules over $\HH^\gr(BU(1))$ and that
$i^*$ hits those generators that come from $\HH^\gr(BU(1))$.
The generators that don't come from $\HH^\gr(BU(1))$ are 
$\HH^\gr(BU(1))$-multiples of $\zeta_0^{-k}\cwd[p]$ or
$\zeta_1^{-k}\cxwd[q]$ for $k>0$, so it suffices to show that these elements
are hit by $i^*$.
But $i^*(\divp) = \cwd[p]$ (because $i^*(x_{p,q}) = 0)$ and $\divp$ is infinitely divisible by $\zeta_0$,
hence
\[
    \cwd[p] = i^*(\divp) = \zeta_0^k i^*(\zeta_0^{-k}\divp),
\]
showing that $\zeta_0^{-k}\cwd[p] = i^*(\zeta_0^{-k}\divp)$.
Similarly, $\zeta_1^{-k}\cxwd[q] = i^*(\zeta_1^{-k}\divq)$.
\end{proof}

Combining this result with Proposition~\ref{prop:binate cohomology}, we get the following.

\begin{corollary}\label{cor:bb additive}
Additively,
\[
    \HH^\gr(\Qpq{2p+1}{2q+1}) \iso \HH^\gr(\Xpq pq) \dirsum \susp^{2(p+q)}\HS_{\GG/e}[\zeta_0^{-1}]
        \dirsum \susp^\nu\HH^\gr(\Xpq pq).
\]
\qed
\end{corollary}

We would like to use this result to give names to the elements in a basis,
but we have to be careful with the first summand. The last summand
is the image of $j_!$; we defined $x_{p,q} = j_!(1)$ and $j_!$ is a map of $\HH^\gr(BU(1))$-modules, so
$j_!(z) = z x_{p,q}$ if $z$ is an element that comes from $BU(1)$.
On the other hand, $\cwd[p]$ is divisible by $\zeta_0$ in the cohomology of $\Xpq pq$, and
\[
    \zeta_0^k j_!(\zeta_0^{-k}\cwd[p]) = j_!(\cwd[p]) = \cwd[p] x_{p,q},
\]
so $\cwd[p] x_{p,q}$ is divisible by $\zeta_0$ with
\[
    \zeta_0^{-k}\cwd[p] x_{p,q} = j_!(\zeta_0^{-k}\cwd[p]).
\]
Similarly, $\zeta_1^{-k}\cxwd[q] x_{p,q} = j_!(\zeta_1^{-k}\cxwd[q])$.
Thus, we are justified in writing the elements of the last summand as $z x_{p,q}$
for $z$ in the cohomology of $\Xpq pq$.

The middle summand we can take as generated by the element $y $ at level $\GG/e$ 
because $i^*(y ) = y $,
as we noted in the proof of Proposition~\ref{prop:bb splitting}.

For the first summand, we are tempted to reuse the names of the elements in
the cohomology of $\Xpq pq$. For those that come from $BU(1)$, that is
perfectly legitimate, given the inclusions $\Xpq pq \includesin \Qpq{2p+1}{2q+1}\includesin BU(1)$. 
However, if $p, q > 0$, then $\cwd[p]$ is divisible by $\zeta_0$ in the cohomology
of $\Xpq pq$ but \emph{not} in the cohomology of $\Qpq{2p+1}{2q+1}$.
(If $p = 0$, every element in the cohomology of $\Qpq{2p+1}{2q+1}$ is divisible by $\zeta_0$;
if $q = 0$, then $\cwd[p] = 0$ in $\Xpq pq$ and its value in $\Qpq{2p+1}{2q+1}$ will
be dealt with later.)
So we cannot write $\zeta_0^{-k}\cwd[p]$, we need to use $\zeta_0^{-k}\divp$ instead.
With this substitution, and substituting $\zeta_1^{-k}\divq$ for $\zeta_1^{-k}\cxwd[q]$,
we could write down a complete basis for the cohomology of $\Qpq{2p+1}{2q+1}$ over $\HS$
(but we will not do so in general, it would be tedious and unenlightening).

It remains to give the multiplicative structure.
We separate out one mildly complicated calculation.
We use here the restriction map $\rho\colon \HH^\gr(X)(\GG/\GG) \to \HH^\gr(X)(\GG/e)$
and the fixed-point map
\[
    (-)^\GG\colon \HH^\alpha(X) \to H^{\alpha_0^\GG}(X^0) \dirsum H^{\alpha_1^\GG}(X^1)
\]
discussed in \cite[\S A.2]{CH:QuadricsI}.

\begin{lemma}\label{lem:divp divq}
In the cohomology of $\Qpq{2p+1}{2q+1}$,
$\divp \divq = \tau(\iota^{2q}\zeta^{p-q}y )$.
\end{lemma}

\begin{proof}
The product $\divp\divq$
lives in grading $p\omega + q\chi\omega$. 
From our calculation of the additive structure, we can see that,
in that grading, all elements are linear combinations of
$\tau(\iota^{2q}\zeta^{p-q}y )$,
$e^{-2}\kappa x_{p,q}$, and $e^{-4}\kappa \zeta_0\cwd x_{p,q}$, that is,
\[
    \divp\divq = a\tau(\iota^{2q}\zeta^{p-q}y ) + be^{-2}\kappa x_{p,q} + ce^{-4}\kappa \zeta_0\cwd x_{p,q}
\]
for some $a,b,c\in \Z$. 
Applying $\rho$, we have
\begin{align*}
    \rho(\divp\divq) &= \iota^{2q}\zeta^{p-q} c^{p+q} \\
    \rho(\tau(\iota^{2q}\zeta^{p-q}y )) &= \iota^{2q}\zeta^{p-q}(1+t)y  = \iota^{2q}\zeta^{p-q}c^{p+q} \\
    \rho(e^{-2}\kappa x_{p,q}) &= 0 \\
    \rho(e^{-4}\kappa \zeta_0\cwd x_{p,q}) &= 0,
\end{align*}
so $a = 1$.
Taking fixed points, we have
\begin{align*}
    (\divp\divq)^\GG &= \bigl[(c^p,1) - 2(cy,0)\bigr]\bigl[(1,c^q) - 2(0,cy)\bigr] \\ &= (0,1)(1,0) = (0,0) \\
    \tau(\iota^{2q}\zeta^{p-q}y )^\GG &= (0,0) \\
    (e^{-2}\kappa x_{p,q})^\GG &= 2(cy,cy) \\
    (e^{-4}\kappa \zeta_0\cwd x_{p,q})^\GG &= 2(0,cy)
\end{align*}
From this we see that we must have $b = 0$ and $c = 0$, proving the lemma.
(The formulas for the fixed points assume that $p>0$ and $q>0$, but the cases where $p=0$ or $q=0$ are simpler,
because $\Qpq{2p+1}{2q+1}^\GG$ then has only one component.)
\end{proof}

\begin{theorem}\label{thm:bb mutliplicative}
As an algebra over $\HH^\gr(BU(1))$, $\HH^\gr(\Qpq{2p+1}{2q+1})$ is generated by
\[
    y  \in \HH^{2(p+q)}(\Qpq{2p+1}{2q+1})(\GG/e)
\]
and 
\[
    x_{p,q} \in \HH^{(p+1)\omega + (q+1)\chi\omega  - 2}(\Qpq{2p+1}{2q+1})(\GG/\GG),
\]
subject to the facts that
\begin{align*}
    \divp &= \cwd[p] - e^{-2(q+1)}\kappa\zeta_1^q x_{p,q} \\
\intertext{is infinitely divisible by $\zeta_0$ and}
    \divq &= \cxwd[q] - e^{-2(p+1)}\kappa\zeta_0^p x_{p,q}
\end{align*}
is infinitely divisible by $\zeta_1$,
and the following relations, where $c = \zeta^{-1}\rho(\cwd)$:
\begin{align*}
    c^{p+q+1} &= 2cy  \\
    y ^2 &= \begin{cases}
                    0 &\text{if $p+q+1$ is even} \\
                    c^{p+q}y  &\text{if $p+q+1$ is odd}
                 \end{cases} \\
    \rho(x_{p,q}) &= \iota^{2(q+1)}\zeta^{p-q} cy  \\
    x_{p,q}^2 &= 0 \\
    \divp \divq &= \tau(\iota^{2q}\zeta^{p-q}y ).
\end{align*}
\end{theorem}

\begin{proof}
We first verify that the stated relations actually hold.
That $\divp$ is infinitely divisible by $\zeta_0$ and $\divq$ is infinitely divisible by $\zeta_1$
was shown in Proposition~\ref{prop:divisible elements}.
The Euler class $\cwd$ restricts to the nonequivariant Euler class $c = e(\omega\dual)$
on applying $\rho$, with a shift in grading accounted for by the factor of $\zeta^{-1}$ stated,
hence $c$ and $y $ satisfy the known relations in $H^*(\Qexp{2(p+q+1)})$.
That $\rho(x_{p,q})$ is $cy $ up to a shift in grading follows from the fact that
both are represented by $\Xp{p+q}$.

The element $x_{p,q}$ can be represented by $j(\Xpq pq)$ and also by the image of
\[
    \Xpq pq \includesin \Spqk pq{p+q+1} \xrightarrow{i} \Qpq{2p+1}{2q+1}.
\]
Because these two representatives are disjoint, $x_{p,q}^2 = 0$.
Note that this and the divisibility of $\divp$ and $\divq$ shows that
$\cwd[p]x_{p,q}$ is infinitely divisible by $\zeta_0$ and $\cxwd[q]x_{p,q}$
is infinitely divisible by $\zeta_1$.

We verified that $\divp\divq = \tau(\iota^{2q}\zeta^{p-q}y )$
in Lemma~\ref{lem:divp divq}.

Here is the structure of the rest of the proof; 
this is similar to the argument used in \cite{CH:QuadricsI} and
will be used for the other types of quadrics as well.
Proposition~\ref{prop:bb splitting} and the discussion after Corollary~\ref{cor:bb additive}
show the following: The principle ideal 
\[
    K^\gr = \rels{x_{p,q}}
    \subset B^\gr = \HH^\gr(\Qpq{2p+1}{2q+1})
\]
satisfies
\begin{align*}
    K^\gr &\iso \susp^\nu \HH^\gr(\Xpq pq) \\
\intertext{and}
    B^\gr/K^\gr &\iso \HH^\gr(\Spqk pq{p+q+1}).
\end{align*}
Let $A^\gr$ be the graded ring
\[
    A^\gr = \HH^\gr(BU(1))[x_{p,q}, y,\zeta_0^{-k}\divp, \zeta_1^{-k}\divq \mid k\geq 1]
\]
where $x_{p,q}$ is a generator of type $\GG/\GG$ in grading $(p+1)\omega + (q+1)\chi\omega - 2$
and $y$ is a generator of type $\GG/e$ in grading $2(p+q)$.
(Precisely, we form the polynomial algebra on these elements and
impose the relations that $\zeta_0\cdot \zeta_0^{-k}\divp = \zeta_0^{-(k-1)}\divp$
and $\zeta_1\cdot\zeta_1^{-k}\divq = \zeta_1^{-(k-1)}\divq$ for $k\geq 1$.)
Let $I^\gr \subset A^\gr$ be the ideal generated by the relations given in the theorem
and let $J^\gr = \rels{x_{p,q}} \subset A^\gr$.

We verified that the relations in $I^\gr$ hold in $B^\gr$, so we have a ring map
$A^\gr/I^\gr \to B^\gr$ taking generators to elements of the same names.
This induces the following commutative diagram 
(we now drop the grading from the notation for readability).
\[
    \xymatrix{
        0 \ar[r] & (J+I)/I \ar[r] \ar[d] & A/I \ar[r] \ar[d] & A/(J+I) \ar[r] \ar[d] & 0 \\
        0 \ar[r] & K \ar[r] & B \ar[r] & B/K \ar[r] & 0
    }
\]
If we can show that 
\begin{align*}
    A/(J+I) &\iso B/K \iso \HH^\gr(\Spqk pq{p+q+1}) \\
\intertext{and}
    (J+I)/I &\iso K \iso \susp^\nu \HH^\gr(\Xpq pq),
\end{align*}
then $A/I \iso B$ as claimed.

Now $A/(J+I)$ is formed by setting $x_{p,q} = 0$
and imposing the relations stated in the theorem.
This makes $\cwd[p] = \divp$ and $\cxwd[q] = \divq$ and makes them divisible by $\zeta_0$ and $\zeta_1$, respectively.
The last relation can be simplified to $\cwd[p]\cxwd[q] = \tau(\iota^{2q}\zeta^{p-q}y )$.
But this is precisely the structure from Proposition~\ref{prop:binate cohomology}, hence
\[
    A/(J+I) \iso \HH^\gr(\Spqk pq{p+q+1})
\]
as claimed.

Now consider $(J+I)/I$ and consider it as a module over $A$. Because $J$ is generated by $x_{p,q}$,
so is $(J+I)/I \iso J/(J\intersect I)$. From the relations in $I$, we get
\begin{align*}
    \cwd[p] x_{p,q} &= \divp x_{p,q}, \\
    \cxwd[q] x_{p,q} &= \divq x_{p,q},
\end{align*}
hence these elements are divisible by $\zeta_0$ and $\zeta_1$, respectively.
We also have
\[
    \cwd[p]\cxwd[q] x_{p,q} = \divp \divq x_{p,q} = \tau(\iota^{2(2q+1)}\zeta^{2(p-q)}cy ^2) = 0
\]
because the nonequivariant relations give $cy ^2 = 0$.
From these we can see that we have exactly the relations we need to say that
\[
    (J+I)/I \iso \susp^\nu \HH^\gr(\Xpq pq)
\]
as claimed.

Therefore, $A/I \iso \HH^\gr(\Qpq{2p+1}{2q+1})$ has the structure asserted in the theorem.
\end{proof}

\begin{remark}\label{rem:bb divp times divq}
Note that, when we apply $\rho$ to $\divp \divq = \tau(\iota^{2q}\zeta^{p-q}y )$
and cancel the unit factors, we get
\[
    c^{p+q} = (1+t)y  = y  + ty .
\]
What this is saying is that, nonequivariantly, the alternative generator $y' = c^{p+q} - y $
is $ty $.
Geometrically, the two classes of maximal isotropic affine subspaces are interchanged by the action of $\GG$,
and this is neatly captured by the structure of the equivariant cohomology considered as Mackey functor-valued.

We can relate this to a distinction we mentioned in \S\ref{sec:mackey}.
The element $y$ generates a summand of $\HS_{\GG/e}$
as an $\HS$-module. That means that the group it lives in, $\HH^{2(p+q)}(\Qpq{2p+1}{2q+1})(\GG/e)$,
is generated \emph{as a group} by $y$ and $ty$.
This explains the difference between our description here of the equivariant cohomology,
where $y$ appears to be the only generator in the ``middle'' grading,
and the description of the nonequivariant cohomology, which has two generators in the middle grading.
\end{remark}

\begin{remark}
There are various other relations implied by the ones given in the theorem. One in particular is
the calculation of $\cwd[p]\cxwd[q]$, for which we need a preliminary calculation:
If $p > 0$, then
\begin{align*}
    \divp \cdot e^{-2(p+1)}\kappa\zeta_0^p x_{p,q}
    &= (\cwd[p] - e^{-2(q+1)}\kappa\zeta_1^q x_{p,q})e^{-2(p+1)}\kappa\zeta_0^p x_{p,q} \\
    &= e^{-2(p+1)}\kappa\zeta_0^p\cwd[p] x_{p,q} \\
    &= e^{-4}\kappa \zeta_0\cwd x_{p,q}
\end{align*}
by induction on the power of $\zeta_0\cwd$.
Similarly, if $q > 0$, then
\[
    \divq\cdot e^{-2(q+1)}\kappa\zeta_1^q x_{p,q} = e^{-4}\kappa \zeta_1\cxwd x_{p,q}.
\]
If $p,q > 0$, this then gives us
\begin{align*}
    \cwd[p]\cxwd[q]
    &= (\divp + e^{-2(q+1)}\kappa\zeta_1^q x_{p,q})(\divq + e^{-2(p+1)}\kappa\zeta_0^p x_{p,q}) \\
    &= \tau(\iota^{2q}\zeta^{p-q}y ) + \divp e^{-2(p+1)}\kappa\zeta_0^p x_{p,q}
        + \divq e^{-2(q+1)}\kappa\zeta_1^q x_{p,q} \\
    &= \tau(\iota^{2q}\zeta^{p-q}y ) + e^{-4}\kappa(\zeta_0\cwd + \zeta_1\cxwd) x_{p,q} \\
    &= \tau(\iota^{2q}\zeta^{p-q}y ) + e^{-2}\kappa x_{p,q}.
\end{align*}
Nonequivariantly this reduces to $c^{p+q} = (1+t)y$, and on fixed sets it gives
\[
    (c^p, c^q) = (2y, 2y).
\]
\end{remark}

\subsection{The $RO(\GG)$-graded part}
Although this structure is best understood using its full grading, it's interesting
to see how it restricts to the $RO(\GG)$ grading, if just to see how complicated it appears
viewed through this narrow lens.
We look just at the case $p\geq q$, the case $p < q$ being similar, by symmetry.

\begin{corollary}\label{cor:bb rog}
If $p\geq 2q$, then $\HH^{RO(\GG)}(\Qpq{2p+1}{2q+1})$ has a basis over $\HS$ given by
\begin{multline*}
    \{ 1,\ \zeta_0\cwd,\ \cwd\cxwd,\ \zeta_0\cwd[2]\cxwd,\ \cwd[2]\cxwd[2], \ldots, \\
        \zeta_0\cwd[q]\cxwd[q-1],\ \cwd[q]\divq,\ \zeta_1^{-1}\cwd[q+1]\divq,\ \ldots,\ \zeta_1^{-(p-q-1)}\cwd[p-1]\divq,\ y , \\
        \zeta_0^{p-q}x_{p,q},\ \zeta_0^{p-q-1}\cxwd x_{p,q}, \ldots,\ \zeta_0^{p-2q+1}\cxwd[q-1]x_{p,q},
        \ \zeta_1^{-(p-2q)}\cxwd[q]x_{p,q}, \\
        \zeta_1^{-(p-2q+1)}\cwd\cxwd[q]x_{p,q},\ldots,
        \ \zeta_1^{-(2p-2q-1)}\cwd[p-1]\cxwd[q]x_{p,q} \}.
\end{multline*}
If $q\leq p < 2q$, then it has a basis given by
\begin{multline*}
    \{ 1,\ \zeta_0\cwd,\ \cwd\cxwd,\ \zeta_0\cwd[2]\cxwd,\ \cwd[2]\cxwd[2], \ldots, \\
        \zeta_0\cwd[q-1]\cxwd[q],\ \cwd[q]\divq,\ \zeta_1^{-1}\cwd[q+1]\divq,\ \ldots,\ \zeta_1^{-(p-q-1)}\cwd[p-1]\divq,\ y , \\
        \zeta_0^{p-q}x_{p,q},\ \zeta_0^{p-q-1}\cxwd x_{p,q}, \ldots,\ \cxwd[p-q]x_{p,q},
        \ \zeta_0\cwd\cxwd[p-q]x_{p,q}, \\
        \cwd\cxwd[p-q+1]x_{p,q},\ \zeta_0\cwd[2]\cxwd[p-q+1]x_{p,q},\ldots,\ \cwd[2q-p]\cxwd[q]x_{p,q},\\
        \zeta_1^{-1}\cwd[2q-p+1]\cxwd[q]x_{p,q},\ldots,\ \zeta_1^{-(2p-2q-1)}\cwd[p-1]\cxwd[q]x_{p,q} \}.
\end{multline*}
In both cases, the basis element $y $ is of type $\GG/e$ while the remainder are of type $\GG/\GG$.
\end{corollary}

We illustrate these bases by showing the locations of the basis elements in $RO(\GG)$ in two cases,
$\Qpq {11}{7}$ and $\Qpq{15}{7}$.
In the following diagrams, a basis element in grading $a + b\sigma$ is shown at $(a,b)$, and the grid lines
are spaced every two.

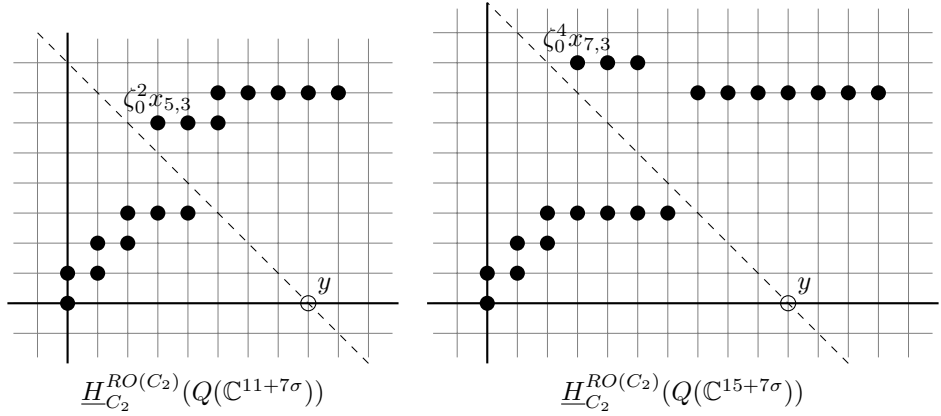
\[
\begin{tikzpicture}[scale=0.4]
	\draw[step=1cm, gray, very thin] (-1.8, -1.8) grid (10.8,8.8);
	\draw[thick] (-2, 0) -- (11, 0);
	\draw[thick] (0, -2) -- (0, 9);

 	\node[below] at (4.5, -2) {$\HH^{RO(\GG)}(\Qexpq{11}{7})$};

    \fill (0,0) circle(0.25cm);
    \fill (0,1) circle(0.25cm);
    \fill (1,1) circle(0.25cm);
    \fill (1,2) circle(0.25cm);
    \fill (2,2) circle(0.25cm);
    \fill (2,3) circle(0.25cm);
    \fill (3,3) circle(0.25cm);
    \fill (4,3) circle(0.25cm);

    \draw (8,0) circle(0.25cm);
    \draw[dashed] (10,-2) -- (-1,9);

    \fill (3,6) circle(0.25cm);
    \fill (4,6) circle(0.25cm);
    \fill (5,6) circle(0.25cm);
    \fill (5,7) circle(0.25cm);
    \fill (6,7) circle(0.25cm);
    \fill (7,7) circle(0.25cm);
    \fill (8,7) circle(0.25cm);
    \fill (9,7) circle(0.25cm);

    \node[above right] at (8,0) {$y$};
    \node[above] at (3,6) {$\zeta_0^2 x_{5,3}$};

\end{tikzpicture}
\quad
\begin{tikzpicture}[scale=0.4]
	\draw[step=1cm, gray, very thin] (-1.8, -1.8) grid (14.8,9.8);
	\draw[thick] (-2, 0) -- (15, 0);
	\draw[thick] (0, -2) -- (0, 10);

 	\node[below] at (6.5, -2) {$\HH^{RO(\GG)}(\Qexpq{15}{7})$};

    \fill (0,0) circle(0.25cm);
    \fill (0,1) circle(0.25cm);
    \fill (1,1) circle(0.25cm);
    \fill (1,2) circle(0.25cm);
    \fill (2,2) circle(0.25cm);
    \fill (2,3) circle(0.25cm);
    \fill (3,3) circle(0.25cm);
    \fill (4,3) circle(0.25cm);
    \fill (5,3) circle(0.25cm);
    \fill (6,3) circle(0.25cm);

    \draw (10,0) circle(0.25cm);
    \draw[dashed] (12,-2) -- (0,10);

    \fill (3,8) circle(0.25cm);
    \fill (4,8) circle(0.25cm);
    \fill (5,8) circle(0.25cm);
    \fill (7,7) circle(0.25cm);
    \fill (8,7) circle(0.25cm);
    \fill (9,7) circle(0.25cm);
    \fill (10,7) circle(0.25cm);
    \fill (11,7) circle(0.25cm);
    \fill (12,7) circle(0.25cm);
    \fill (13,7) circle(0.25cm);

    \node[above right] at (10,0) {$y$};
    \node[above] at (3,8) {$\zeta_0^4 x_{7,3}$};

\end{tikzpicture}
\]
The solid dots indicate basis elements of type $\GG/\GG$ while the open dot
is the basis element $y$ of type $\GG/e$. Note that we could put a basis element
of type $\GG/e$ anywhere along the dotted line by replacing $y$ by
$\iota^k y$ for any $k\in \Z$, so its exact position on that line is arbitrary.

To describe the multiplicative structure of $\HH^{RO(\GG)}(\Qpq{2p+1}{2q+1})$
purely in terms of this basis would be messy, to say the least.
But calculations in this ring are easily done viewing it as a subring of the one
graded on $RO(\Pi BU(1))$ described
in Theorem~\ref{thm:bb mutliplicative}.
This is similar to the comparison we gave in \cite[\S5.2]{CHTFiniteProjSpace}
of our calculation of the cohomology of finite projective spaces to
Gaunce Lewis's $RO(\GG)$-graded structure in \cite{LewisCP}, but 
the contrast would be even starker.

\section{Quadrics of type (B,D) and (D,B)}

Because there is a $\GG$-diffeomorphism $\Qpq{2p}{2q+1} \homeo \Qpq{2q+1}{2p}$,
the cohomology of quadrics of type (B,D) follows from that of type (D,B),
so we concentrate here on the latter.
As mentioned earlier, we will grade all cohomology on $RO(\Pi BU(1))$,
ignoring for now the fact that $RO(\Pi\Qpq{2}{2q+1})$ is generally larger;
we will return to deal with that in the third paper in this series.

Because $\Qpq 0{2q+1}$ has trivial $\GG$-action and free nonequivariant cohomology, we get
\[
    \HH^\gr(\Qpq 0{2q+1}) \iso H^*(\eQp{2q+1})\tensor \HS[\zeta_0^{\pm 1}].
\]
We will assume for the remainder of this section that $p>0$.

Once again, our calculations are based on a cofibration.
Define
\begin{alignat*}{2}
    i\colon \Xpq pq &\includesin \Qpq{2p}{2q+1}, &\qquad i[\vec x:\vec y] &= [\vec x:\vec 0:\vec y:0:\vec 0] \\
\intertext{and, as in (\ref{def:j}), let}
    j\colon \Xpq pq &\includesin \Qpq{2p}{2q+1}, & i[\vec u:\vec v] &= [\vec 0:\vec u:\vec 0:0:\vec v].
\end{alignat*}
The normal bundle to $j$ is now
\[
    \nu = (2p - p)\omega \dirsum (2q+1 - q)\chi\omega\dual - O(2) = p\omega \dirsum (q+1)\chi\omega - O(2)
\]
which has dimension $p\omega + (q+1)\chi\omega - 2$, which we will also write as $\nu$.

As in the (B,B) case, we define
\[
    x_{p,q} = [j(\Xpq pq)]^* \in \HH^\nu(\Qpq{2p}{2q+1}).
\]

\begin{proposition}\label{prop:bd cofibration}
We have a cofibration sequence
\[
    \Xpq pq_+ \xrightarrow{i} (\Qpq{2p}{2q+1})_+ \to \susp^\nu j(\Xpq pq)_+
\]
where $\nu$ is the normal bundle to $j$.
\end{proposition}

\begin{proof}
The projection
\[
   \Qpq{2p}{2q+1} \setminus i(\Xpq pq) \to j(\Xpq pq),\quad[\vec x:\vec u:\vec y:w:\vec v] \mapsto [\vec 0:\vec u:\vec 0:0:\vec v]
\]
can be identified with the normal bundle $\nu$.
\end{proof}

This implies a long exact sequences in cohomology from which we want to get a split short exact sequence.
In order to do that, we first notice the following. Let
\[
    k\colon \Qpq{2p-1}{2q+1} \to \Qpq{2p}{2q+1}
\]
be the inclusion defined by
\[
    k[\vec x:z:\vec u:\vec y:w:\vec v] = [\vec x:z:z:\vec u:\vec y:w:\vec v].
\]

\begin{proposition}\label{prop:db divisible}
In $\HH^\gr(\Qpq{2p}{2q+1})$,
\[
    \divp = \divp[2p] = \divp[2p,2q+1] := k_!(\divp[2p-1,2q+1]) = \cwd[p] - e^{-2(q+1)}\kappa\zeta_1^q\cwd x_{p,q}
\]
is infinitely divisible by $\zeta_0$, and
\[
    \divq = \divq[2q+1] = \divq[2p,2q+1] := \cxwd[q] - e^{-2p}\kappa\zeta_0^{p-1} x_{p,q}
\]
is infinitely divisible by $\zeta_1$.
\end{proposition}

Note that the formula for $\divp[2p]$
has an extra factor of $\cwd$ in its second term, compared to the formula for $\divp[2p+1]$.
Likewise, for quadrics of type (B,D), the formula for $\divq[2q]$ would have
an extra factor of $\cxwd$ in its second term.

\begin{proof}
Consider the section of $\omega\dual$ given by $x_p - u_p$ in the usual coordinates
on $\Qpq{2p}{2q+1}$. The zero locus of this section is exactly $k(\Qpq{2p-1}{2q+1})$,
hence
\[
    \cwd = k_!(1) = [\Qpq {2p-1}{2q+1}]^*.
\]
We also have that $\cwd x_{p,q} = j_!(\cwd)$ is represented by $\Xpq{p-1}{q}$, as is
$k_!(x_{p-1,q})$, hence
\[
    k_!(x_{p-1,q}) = \cwd x_{p,q}.
\]
From this it follows that
\[
    k_!(\divp[2p-1]) = k_!(\cwd[p-1] - e^{-2(q+1)}\kappa\zeta_1^q x_{p-1,q})
    = \cwd[p] - e^{-2(q+1)}\kappa\zeta_1^q\cwd x_{p,q},
\]
verifying the alimed value of this pushforward, which we take as the definition of $\divp[2p]$.
We know that $\divp[2p-1]$ is infinitely divisible by $\zeta_0$ in the cohomology
of $\Qpq{2p-1}{2q+1}$, hence the same is true for its pushforward $\divp[2p]$.

That $\divq[2q+1]$ is infinitely divisible by $\zeta_1$ 
follows from Porposition~\ref{prop:general divisibility} by the same argument as in
Proposition~\ref{prop:divisible elements}.
\end{proof}

\begin{remark}\label{rem:db divp simplifies}
When $p = 1$, the formula for $\divp[2]$ simplifies:
We know that $\cwd x_{1,q} = j_!(\cwd)$ is divisible by $\zeta_0$
because $\cwd$ is divisible by $\zeta_0$ in the cohomology of $\Xpq 1q$. Hence,
\[
    \divp[2] = \cwd - e^{-2(q+1)}\kappa\zeta_1^q\cwd x_{1,q}
    = \cwd - e^{-2(q+1)}\kappa\xi\zeta_1^{q-1}\zeta_0^{-1}\cwd x_{1,q}
    = \cwd
\]
because $e^{-2(q+1)}\kappa\xi = 0$.
(In the case $q = 0$, we are also using that $\zeta_1$ is invertible in the cohomology
of $\Xp{}$.)
\end{remark}

\begin{proposition}\label{prop:db splitting}
There is a split short exact sequence
\[
    0 \to
    \susp^\nu \HH^\gr(\Xpq pq) \xrightarrow{j_!}
    \HH^\gr(\Qpq{2p}{2q+1}) \xrightarrow{i^*}
    \HH^\gr(\Xpq pq)
    \to 0
\]
where $\nu = p\omega + (q+1)\chi\omega - 2$.
\end{proposition}

\begin{proof}
We need to show that $i^*$ is surjective.
As in Proposition~\ref{prop:bb splitting}, this comes down to showing that the
elements $\zeta_0^{-k}\cwd[p]$ and $\zeta_1^{-k}\cxwd[q]$, $k\geq 1$,
are in the image of $i^*$,
which follows from the existence of the divisible elements $\divp$ and $\divq$.
\end{proof}

\begin{corollary}\label{cor:db additive}
Additively,
\[
    \HH^\gr(\Qpq{2p}{2q+1}) \iso \HH^\gr(\Xpq pq) \dirsum \susp^\nu\HH^\gr(\Xpq pq)
\]
where $\nu = p\omega + (q+1)\chi\omega - 2$.
\qed
\end{corollary}

We complete the picture with the multiplicative structure.
We start once again with the computation of $\divp\divq$.

\begin{lemma}\label{lem:db divp divq}
In the cohomology of $\Qpq{2p}{2q+1}$, $\divp\divq = \tau(\iota^{-2})\zeta_1 x_{p,q}$.
\end{lemma}

\begin{proof}
The product $\divp\divq$ lies in grading $p\omega + q\chi\omega$.
From the additive calculation we have done so far, we can see that,
in that grading, all elements are linear combinations of
$\tau(\iota^{-2})\zeta_1 x_{p,q}$, $e^{-2}\kappa\cwd x_{p,q}$, and $e^{-4}\kappa\zeta_0\cwd[2]x_{p,q}$,
that is,
\[
    \divp\divq = a\tau(\iota^{-2})\zeta_1 x_{p,q} + be^{-2}\kappa\cwd x_{p,q} + ce^{-4}\kappa\zeta_0\cwd[2]x_{p,q}
\]
for some $a,b,c\in \Z$.
(If $p = 1$ or $q = 0$, the third term does not appear; if both $p=1$ and $q=0$, only the
first term appears. The proof is similar in those cases.)
To apply $\rho$, we first have that $\rho(x_{p,q}) = \iota^{2(q+1)}\zeta^{p-q-1}y$ where
$y$ is the generator of the nonequivariant cohomology from Proposition~\ref{prop:nonequivariant},
as we can see by comparing representatives and gradings.
Applying $\rho$ to the elements of interest to us, and writing $\zeta = \rho(\zeta_1)$, we have
\begin{align*}
    \rho(\divp\divq) &= \iota^{2q}\zeta^{p-q}c^{p+q} \\
    \rho(\tau(\iota^{-2})\zeta_1 x_{p,q}) &= 2\iota^{2q}\zeta^{p-q}y = \iota^{2q}\zeta^{p-q}c^{p+q} \\
    \rho(e^{-2}\kappa\cwd x_{p,q}) &= 0 \\
    \rho(e^{-4}\kappa\zeta_0\cwd[2]x_{p,q}) &= 0.
\end{align*}
Thus, $a = 1$. Taking fixed points, we have
\begin{align*}
    (\divp\divq)^\GG &= \bigl[(c^p,1) - 2(cy,0) \bigr] \bigl[ (1,c^q) - 2(0,y)\bigr] \\
        &= (0,1)(1,0) = (0,0) \\
    (\tau(\iota^{-2})\zeta_1 x_{p,q})^\GG &= (0,0) \\
    (e^{-2}\kappa\cwd x_{p,q})^\GG &= 2(cy, y) \\
    (e^{-4}\kappa\zeta_0\cwd[2]x_{p,q})^\GG &= 2(0,y),
\end{align*}
hence $b = c = 0$.
\end{proof}

\begin{theorem}\label{thm:db multiplicative}
As an algebra over $\HH^{\gr}(BU(1))$, 
the cohomology $\HH^\gr(\Qpq{2p}{2q+1})$
is generated by
\[
    x_{p,q} \in \HH^{p\omega + (q+1)\chi\omega - 2}(\Qpq{2p}{2q+1})
\]
subject to the facts that
\begin{align*}
    \divp &= 
    \begin{cases}
        \cwd & \text{if $p=1$} \\
        \cwd[p] - e^{-2(q+1)}\kappa\zeta_1^q\cwd x_{p,q} & \text{if $p>1$}
    \end{cases} \\
\intertext{is infinitely divisible by $\zeta_0$ and}
    \divq &= \cxwd[q] - e^{-2p}\kappa\zeta_0^{p-1} x_{p,q} 
\end{align*}
is infinitely divisible by $\zeta_1$,
and the following relations:
\begin{align*}
    x_{p,q}^2 &= \begin{cases}
                    0 & \text{if $p$ is even} \\
                    e^2\cwd[p-1]\cxwd[q]x_{p,q} & \text{if $p$ is odd}
                 \end{cases}
    \\
    \divp\divq &= \tau(\iota^{-2})\zeta_1  x_{p,q}.
\end{align*}
\end{theorem}

Note that, when $p=1$, we define $\divp[2] = \cwd$, so that we get the assumption that
$\cwd$ is infinitely divisible by $\zeta_0$. But, as in Remark~\ref{rem:db divp simplifies},
that implies that $\divp[2] = \cwd - e^{-2(q+1)}\kappa\zeta_1^q\cwd x_{1,q}$ as well.

\begin{proof}
We have already shown the divisibility of $\divp$ and $\divq$ and the relation involving
their product. The relation remaining to check is the one for $x_{p,q}^2$.
We could give a geometric proof but give an algebraic one instead.

$x_{p,q}^2$ lives in grading $2p\omega + 2(q+1)\chi\omega - 4$.
From the additive calculation, it must be an integer multiple of $e^2\cwd[p-1]\cxwd[q]x_{p,q}$.
These elements vanish on applying $\rho$, so we look at fixed points:
\begin{align*}
    (x_{p,q}^2)^\GG &= (y^2, y^2) 
        = \begin{cases}
            (0,0) & \text{if $p$ is even} \\
            (c^{p-1}y, 0) & \text{if $p$ is odd}
          \end{cases} \\
    (e^2\cwd[p-1]\cxwd[q]x_{p,q})^\GG &= (c^{p-1}y, c^qy) = (c^{p-1}y, 0)
\end{align*}
From this we can see that $x_{p,q}^2$ has the form claimed in the theorem.

We now use the same general structure for the remainder of the proof as was used
in the proof of Theorem~\ref{thm:bb mutliplicative}.
Let
\[
    K = \rels{x_{p,q}} 
    \subset B = \HH^\gr(\Qpq{2p}{2q+1}).
\]
Proposition~\ref{prop:db splitting} implies that
\begin{align*}
    K &\iso \susp^\nu\HH^\gr(\Xpq pq) \\
\intertext{and}
    B/K &\iso \HH^\gr(\Xpq pq).
\end{align*}
Let $A$ be the graded ring
\[
    A = \HH^\gr(BU(1))[x_{p,q}, \zeta_0^{-k}\divp, \zeta_1^{-k}\divq \mid k\geq 1],
\]
in which we have that $\cwd[p]x_{p,q}$ is infinitely divisible by $\zeta_0$ and
$\cxwd[q]x_{p,q}$ is infinitely divisible by $\zeta_1$,
let $I\subset A$ be the ideal generated by the relations in the theorem, and let
$J = \rels{x_{p,q}}\subset A$.

We have already checked that the relations in $I$ hold in $B$, so we
have a map $A/I \to B$ taking generators to elements of the same name,
and we get the following commutative diagram.
\[
    \xymatrix{
        0 \ar[r] & (J+I)/I \ar[r] \ar[d] & A/I \ar[r] \ar[d] & A/(J+I) \ar[r] \ar[d] & 0 \\
        0 \ar[r] & K \ar[r] & B \ar[r] & B/K \ar[r] & 0
    }
\]
$A/(J+I)$ is formed by setting $x_{p,q}$ to $0$ and then imposing the relations of the theorem,
of which the only one remaining becomes $\cwd[p]\cxwd[q] = 0$. This is exactly the presentation
of $\HH^\gr(\Xpq pq)$, so we have $A/(J+I) \iso B/K$.

To check that $(J+I)/I \iso \susp^\nu\HH^\gr(\Xpq pq)$, 
we first show that the relations in $I$ imply that $\cwd[p]\cxwd[q]x_{p,q} = 0$.
We have
\begin{align*}
    \divp\divq x_{p,q} &= \tau(\iota^{-2})\zeta_1 x_{p,q}^2 \\
    &= \begin{cases}
        0 & \text{if $p$ is even} \\
        \tau(\iota^{-2})\zeta_1 \cdot e^2\cwd[p-1]\cxwd[q]x_{p,q} & \text{if $p$ is odd}
       \end{cases} \\
    &= 0
\end{align*}
because $\tau(\iota^{-2})e^2 = \tau(\iota^{-2}\rho(e^2)) = 0$.
If $p = 1$, so $\divp = \cwd$, expanding the definition of $\divq$ and rearranging gives
\begin{align*}
    \cwd\cxwd[q]x_{p,q} &= e^{-2}\kappa\cwd x_{p,q}^2 \\
    &= \kappa\cwd\cxwd[q] x_{p,q},
\end{align*}
hence $(1-\kappa)\cwd\cxwd[q]x_{p,q} = 0$, which implies that $\cwd\cxwd[q] x_{p,q} = 0$ because
$1-\kappa$ is a unit.
If $p > 1$, expanding the definitions of $\divp$ and $\divq$ and rearranging gives
\begin{align*}
    \cwd[p]\cxwd[q]x_{p,q}
    &= \bigl(e^{-2p}\kappa\zeta_0^{p-1}\cwd[p] + e^{-2(q+1)}\kappa\zeta_1^q\cwd\cxwd[q]\bigr)x^2_{p,q} \\
    &\qquad\qquad    + 2e^{-2(p+q+1)}\kappa\zeta_0^{p-1}\zeta_1^q\cwd x_{p,q}^3 \\
    &= 0 \\
\intertext{if $p$ is even, but if $p$ is odd we get instead}
    &= \cwd[p]\cxwd[q]x_{p,q}(e^{-2(p-1)}\kappa\zeta_0^{p-1}\cwd[p-1] + e^{-2q}\kappa\zeta_1^q\cxwd[q]) \\
\intertext{(this uses that $e^{-2(p+q+1)}\kappa\xi = 0$)}
    &= \cwd[p]\cxwd[q]x_{p,q}\cdot e^{-2}\kappa(\zeta_0\cwd + \zeta_1\cxwd) \\
    &= \cwd[p]\cxwd[q]x_{p,q}\cdot e^{-2}\kappa(g\zeta_0\cwd + e^2) \\
    &= \kappa\cwd[p]\cxwd[q]x_{p,q},
\end{align*}
hence $(1-\kappa)\cwd[p]\cxwd[q] x_{p,q} = 0$ and again we may conclude that $\cwd[p]\cxwd[q] x_{p,q} = 0$.

We now have in $(J+I)/I$ that
\begin{align*}
    \divp x_{p,q} &= (\cwd[p] - e^{-2(q+1)}\kappa\zeta_1^q\cwd x_{p,q})x_{p,q} \\
    &= \begin{cases}
         \cwd[p]x_{p,q} & \text{if $p$ is even} \\
         \cwd[p]x_{p,q} - e^{-2q}\kappa\zeta_1^q\cwd[p]\cxwd[q] x_{p,q} & \text{if $p$ is odd}
       \end{cases} \\
    &= \cwd[p]x_{p,q}
\end{align*}
and
\begin{align*}
    \divq x_{p,q} &= (\cxwd[q] - e^{-2p}\kappa\zeta_0^{p-1} x_{p,q})x_{p,q} \\
    &= \begin{cases}
         \cxwd[q]x_{p,q} & \text{if $p$ is even} \\
         \cxwd[q]x_{p,q} - e^{-2(p-1)}\kappa\zeta_0^{p-1}\cwd[p-1]\cxwd[q] x_{p,q} & \text{if $p$ is odd}
       \end{cases} \\
    &= \begin{cases}
        \cxwd[q] x_{p,q} &\text{if $p$ is even} \\
        (1-\kappa)\cxwd[q] x_{p,q} &\text{if $p=1$} \\
        (1-\epsilon)\cxwd[q] x_{p,q} &\text{if $p>1$ is odd}
       \end{cases}
\end{align*}
where $\epsilon = e^{-2}\kappa\zeta_0\cwd$ and $1-\epsilon$ is a unit
(but see the remark following this proof).
These two calculations imply that $\cwd[p]x_{p,q}$ is divisible by $\zeta_0$ and
$\cxwd[q]x_{p,q}$ is divisible by $\zeta_1$, by the assumptions on $\divp$ and $\divq$.
We therefore have the relations that say that
$(J+I)/I \iso \susp^\nu\HH^\gr(\Xpq pq)$ as claimed.

Therefore, $A/I \iso \HH^\gr(\Qpq{2p}{2q+1})$ has the structure asserted in the theorem.
\end{proof}

\begin{remark}
We showed in the proof above that the relations imply that $\divq x_{p,q} = (1-\epsilon)\cxwd[q]x_{p,q}$
when $p>1$ is odd, from which we deduced that $\cxwd[q]x_{p,q}$ is divisible by $\zeta_1$. But once we know
that divisibility, we have
\begin{align*}
    (1-\epsilon)\cxwd[q]x_{p,q}
    &= (1-e^{-2}\kappa\zeta_0\cwd)\cxwd[q]x_{p,q} \\
    &= \cxwd[q]x_{p,q} - e^{-2}\kappa\xi\zeta_1^{-1}\cwd\cxwd[q]x_{p,q} \\
    &= \cxwd[q]x_{p,q}.
\end{align*}
The case $p=1$ cannot be so rewritten.
\end{remark}

\begin{remark}\label{rem:db cwd p cxwd q}
As we did in the (B,B) case, we can compute $\cwd[p]\cxwd[q]$. The computation has various cases,
depending on the parity of $p$ and whether it is 1 or larger, but all cases have the same result:
\[
    \cwd[p]\cxwd[q] = \bigl(\tau(\iota^{-2})\zeta_1 + e^{-2}\kappa\cwd\bigr)x_{p,q}.
\]
The calculation is otherwise similar to that in Remark~\ref{rem:bb divp times divq}.
When we apply $\rho$, this becomes
\[
    c^{p+q} = 2y \qquad\text{in } H^{2(p+q)}(\eQp{2(p+q)+1}).
\]
When we take fixed points it becomes
\[
    (c^p, c^q) = (2cy, 2y) \qquad\text{in } H^{2p}(\eQp{2p})\dirsum H^{2q}(\eQp{2q+1}).
\]
\end{remark}

\section{Quadrics of type (D,D)}

Finally, we look at quadrics of the form $\Qpq{2p}{2q}$.
Once again, we will restrict to
grading on $RO(\Pi BU(1))$ even though $RO(\Pi\Qpq{2p}{2q})$ may be larger,
and return to the exceptional cases $p=1$ or $q=1$ in the sequel to this paper.

Define
\begin{align*}
    i\colon \Xpq pq &\to \Qpq{2p}{2q}, & i[\vec x:\vec y] &= [\vec x:\vec 0:\vec y:\vec 0] \\
\intertext{and, as in (\ref{def:j}),}
    j\colon \Xpq pq &\to \Qpq{2p}{2q}, & j[\vec u:\vec v] &= [\vec 0:\vec u:\vec 0:\vec v].
\end{align*}
The normal bundle to $j$ is $\nu = p\omega\dual\dirsum q\chi\omega\dual - O(2)$, which, as a dimension,
is $\nu = p\omega + q\chi\omega - 2$.
Write
\[
    x_{p,q} = [j(\Xpq pq)]^* \in \HH^\nu(\Qpq{2p}{2q}).
\]

\begin{proposition}\label{prop:dd cofibration}
We have a cofibration sequence
\[
    \Xpq pq_+ \xrightarrow{i} (\Qpq{2p}{2q})_+ \to \susp^\nu j(\Xpq pq)_+
\]
where $\nu$ is the normal bundle to $j$.
\end{proposition}

\begin{proof}
The projection
\[
    \Qpq{2p}{2q}\setminus i(\Xpq pq) \to j(\Xpq pq), \quad
    [\vec x:\vec u:\vec y:\vec v] \mapsto [\vec 0:\vec u:\vec 0:\vec v]
\]
can be identified with the normal bundle $\nu$.
\end{proof}

We have two inclusions from smaller quadrics we consider. Let
\[
    k\colon \Qpq{2p-1}{2q} \to \Qpq{2p}{2q}
\]
be defined by
\[
    k[\vec x:z:\vec u:\vec y:\vec v] = [\vec x:z:z:\vec u:\vec y:\vec v]
\]
and let
\[
    \ell\colon \Qpq{2p}{2q-1} \to \Qpq{2p}{2q}
\]
be defined by
\[
    \ell[\vec x:\vec u:\vec y:w:\vec v] = [\vec x:\vec u:\vec y:w:w:\vec v].
\]

\begin{proposition}\label{prop:dd divisible}
In $\HH^\gr(\Qpq{2p}{2q})$, the element
\[
    \divp = \divp[2p] = \divp[2p,2q] := k_!(\divp[2p-1,2q]) = \cwd[p] - e^{-2q}\kappa\zeta_1^{q-1}\cwd x_{p,q}
\]
is infinitely divisible by $\zeta_0$. Similarly, the element
\[
    \divq = \divq[2q] = \divq[2p,2q] := \ell_!(\divq[2p,2q-1]) = \cxwd[q] - e^{-2p}\kappa\zeta_0^{p-1}\cxwd x_{p,q}
\]
is infinitely divisible by $\zeta_1$.
\end{proposition}

\begin{proof}
The proof that $k_!(\divp[2p-1,2q])$ has the form claimed is similar to the argument in the proof of
Proposition~\ref{prop:db divisible}. The divisibility of $\divp[2p-1,2q]$ then implies the divisibility of $\divp[2p,2q]$.
The argument for $\divq[2p,2q]$ is similar.
\end{proof}

\begin{remark}\label{rem:dd divp simplifies}
As in the (D,B) case, the formulas simplify if $p = 1$ or $q = 1$ and the other is larger:
\begin{alignat*}{2}
    \divp[2,2q] &= \cwd   &\qquad&\text{if $q > 1$} \\
    \divq[2p,2] &= \cxwd &&\text{if $p > 1$.}
\end{alignat*}
In the case $p = q = 1$, we have
\begin{align*}
    \divp[2,2] &= \cwd - e^{-2}\kappa \cwd x_{1,1} = \cwd(1 - e^{-2}\kappa x_{1,1}) \\
    \divq[2,2] &= \cxwd - e^{-2}\kappa \cxwd x_{1,1} = \cxwd(1 - e^{-2}\kappa x_{1,1})
\end{align*}
and these don't simplify. 
On the other hand, in Remark~\ref{rem:dd unit}, we will show that the factor $1 - e^{-2}\kappa x_{1,1}$ is a unit,
in fact is its own inverse,
so the divisibility of $\divp[2,2]$ and $\divq[2,2]$ imply the same for $\cwd$ and $\cxwd$.
This divisibility also follows from the calculation $\eta_0(\cwd) = 0$,
which is true because it is the Euler class of a bundle over the discrete
space $\Qexp 2$ with fibers $\C$, and similarly for $\cxwd$.
\end{remark}

\begin{proposition}\label{prop:dd splitting}
There is a split short exact sequence
\[
    0 \to \susp^\nu\HH^\gr(\Xpq pq) \xrightarrow{j_!}
    \HH^\gr(\Qpq{2p}{2q}) \xrightarrow{i^*}
    \HH^\gr(\Xpq pq) \to 0
\]
where $\nu = p\omega + q\chi\omega - 2$.
\end{proposition}

\begin{proof}
The proof is essentially the same as the proof of Proposition~\ref{prop:db splitting},
this time using the divisibility of $\divp[2p]$ and $\divq[2q]$.
\end{proof}

\begin{corollary}\label{cor:dd additive}
Additively,
\[
    \HH^\gr(\Qpq{2p}{2q}) \iso \HH^\gr(\Xpq pq) \dirsum \susp^\nu\HH^\gr(\Xpq pq)
\]
where $\nu = p\omega + q\chi\omega - 2$.
\qed
\end{corollary}

The significant difference between this and Corollary~\ref{cor:db additive} is
in $\nu$, which was $p\omega + (q+1)\chi\omega - 2$ in the earlier corollary.
The value of $\nu$ in Corollary~\ref{cor:dd additive}
produces two basis elements reducing to elements with the same
nonequivariant grading, to match what we know has to happen in 
the nonequivariant cohomology of quadrics of type~D.

Finally, the multiplicative structure. 
In the process, we redefine $\divp[2,2]$ and $\divq[2,2]$, which will be harmless
by Remark~\ref{rem:dd divp simplifies}.

\begin{theorem}\label{thm:dd multiplicative}
As an algebra over $\HH^{RO(\Pi B)}(BU(1))$, $\HH^\gr(\Qpq{2p}{2q})$ is generated by
\[
    x_{p,q} \in \HH^{p\omega + q\chi\omega - 2}(\Qpq{2p}{2q})
\]
subject to the facts that
\begin{alignat*}{2}
    \divp &= 
    \begin{cases}
        \cwd &\text{if $p=1$} \\
        \cwd[p] - e^{-2q}\kappa\zeta_1^{q-1}\cwd x_{p,q} &\text{if $p>1$}
    \end{cases} \\
\intertext{is infinitely divisible by $\zeta_0$ and}
    \divq &= 
    \begin{cases}
        \cxwd &\text{if $q=1$} \\
        \cxwd[q] - e^{-2p}\kappa\zeta_0^{p-1}\cxwd x_{p,q} &\text{if $q > 1$}
    \end{cases}
\end{alignat*}
is infinitely divisible by $\zeta_1$,
and the following relations:
\begin{align*}
    x_{p,q}^2 &= \begin{cases}
                    0 & \text{if $p$ and $q$ are even} \\
                    e^2\cwd[p-1]\cxwd[q-1]x_{p,q} & \text{if $p$ and $q$ are odd} \\
                    \zeta_0\cwd[p]\cxwd[q-1]x_{p,q} & \text{if $p$ is even and $q$ is odd} \\
                    \zeta_1\cwd[p-1]\cxwd[q]x_{p,q} & \text{if $p$ is odd and $q$ is even}
                 \end{cases} \\
    \divp \divq &= \tau(\iota^{-2})\zeta_0\cwd x_{p,q}.
\end{align*}
\end{theorem}

\begin{proof}
The proof is very similar to that of Theorem~\ref{thm:db multiplicative}.
We will check here only that the two stated relations hold, and leave the rest of the
proof to the reader.

To calculate $\divp\divq$, we note that it lives in grading
$p\omega + q\chi\omega$. From the additive calculation, all elements
in that grading are given by integer linear combinations of
$\tau(\iota^{-2})\zeta_0\cwd x_{p,q}$,
$e^{-2}\kappa\cwd\cxwd x_{p,q}$, and $e^{-4}\kappa\zeta_0\cwd[2]\cxwd x_{p,q}$, so
\[
    \divp\divq = \alpha \tau(\iota^{-2})\zeta_0\cwd x_{p,q} + \beta e^{-2}\kappa\cwd\cxwd x_{p,q} 
    + \gamma e^{-4}\kappa\zeta_0\cwd[2]\cxwd x_{p,q}
\]
for some integers $\alpha$, $\beta$, and $\gamma$.
Applying $\rho$, we get
\begin{align*}
    \rho(\divp\divq) &= c^{p+q} = 2cy \\
    \rho(\tau(\iota^{-2})\zeta_0\cwd x_{p,q}) &= 2cy \\
    \rho(e^{-2}\kappa\cwd\cxwd x_{p,q}) &= 0 \\
    \rho(e^{-4}\kappa\zeta_0\cwd[2]\cxwd x_{p,q}) &= 0
\end{align*}
so $\alpha = 1$.
Taking fixed points, we get
\begin{align*}
    (\divp\divq)^\GG &= \bigl[(c^p,1) - 2(cy,0)\bigr]\bigl[ (1,c^q) - 2(0,cy) \bigr] \\
        &= (0,0) \\
    (\tau(\iota^{-2})\zeta_0\cwd x_{p,q})^\GG &= 0 \\
    (e^{-2}\kappa\cwd\cxwd x_{p,q})^\GG &= (2cy,2cy) \\
    (e^{-4}\kappa\zeta_0\cwd[2]\cxwd x_{p,q})^\GG &= (0,2cy)
\end{align*}
so $\beta = \gamma = 0$. Therefore, $\divp\divq = \tau(\iota^{-2})\zeta_0\cwd x_{p,q}$ as claimed.
(The calculations above assume that $p,q>1$, but the special cases where $p=1$ or $q=1$ are similar.)

The element $x_{p,q}^2$ lives in grading $2p\omega + 2q\chi\omega - 4$,
where all elements are linear combinations
of $\zeta_0\cwd[p]\cxwd[q-1]x_{p,q}$ and $e^2\cwd[p-1]\cxwd[q-1]x_{p,q}$, that is,
\[
    x_{p,q}^2 = \alpha \zeta_0\cwd[p]\cxwd[q-1]x_{p,q} + \beta e^2\cwd[p-1]\cxwd[q-1]x_{p,q}
\]
for some $\alpha\in A(\GG)$ and $beta\in\Z$.
Applying $\rho$, we get
\begin{align*}
    \rho(x_{p,q}^2) &= y^2
        = \begin{cases}
            0 & \text{if $p+q$ is even} \\
            c^{p+q-1}y & \text{if $p+q$ is odd}
          \end{cases} \\
    \rho(\zeta_0\cwd[p]\cxwd[q-1]x_{p,q}) &= c^{p+q-1}y \\
    \rho(e^2\cwd[p-1]\cxwd[q-1]x_{p,q}) &= 0,
\end{align*}
so $\rho(\alpha) = 0$ if $p+q$ is even and $1$ if $p+q$ is odd.
Taking fixed points, we get
\begin{align*}
    (x_{p,q}^2)^\GG &= (y^2,y^2) \\
        &= \begin{cases}
            (0,0) & \text{if $p$ and $q$ are even} \\
            (c^{p-1}y, c^{q-1}y) & \text{if $p$ and $q$ are odd} \\
            (0, c^{q-1}y) & \text{if $p$ is even and $q$ is odd} \\
            (c^{p-1}y, 0) & \text{if $p$ is odd and $q$ is even}
           \end{cases} \\
    (\zeta_0\cwd[p]\cxwd[q-1]x_{p,q})^\GG &= (0, c^{q-1}y) \\
    (e^2\cwd[p-1]\cxwd[q-1]x_{p,q})^\GG &= (c^{p-1}y, c^{q-1}y).
\end{align*}
From this we get
\[
    (\alpha^\GG,\beta) =
    \begin{cases}
        (0,0) &\text{if $p$ and $q$ are even} \\
        (0,1) &\text{if $p$ and $q$ are odd} \\
        (1,0) &\text{if $p$ is even and $q$ is odd} \\
        (-1,1) &\text{if $p$ is odd and $q$ is even,}
    \end{cases}
\]
which tells us that
\[
    \alpha =
    \begin{cases}
        0 & \text{if $p$ and $q$ have the same parity} \\
        1 & \text{if $p$ is even and $q$ is odd} \\
        1-\kappa & \text{if $p$ is odd and $q$ is even}
    \end{cases}
\]
This gives us the relation stated in the theorem after we note in the last case that
we can rewrite
\[
    (1-\kappa) \zeta_0\cwd[p]\cxwd[q-1]x_{p,q} + e^2\cwd[p-1]\cxwd[q-1]x_{p,q}
    = \zeta_1\cwd[p-1]\cxwd[q]x_{p,q}.
    \qedhere
\]
\end{proof}

\begin{remark}\label{rem:dd unit}
In the case $p = q = 1$, we have
\begin{align*}
    (1 - e^{-2}\kappa x_{1,1})^2
    &= 1 - 2e^{-2}\kappa x_{1,1} + 2e^{-4}\kappa x_{1,1}^2 \\
    &= 1 - 2e^{-2}\kappa x_{1,1} + 2e^{-2}\kappa x_{1,1} \\
    &= 1.
\end{align*}
Hence, $1-e^{-2}\kappa x_{1,1}$ is a unit in the cohomology of $\Qpq 22$.
This justifies our redefinition of $\divp[2,2]$ and $\divq[2,2]$ as $\cwd$ and $\cxwd$, respectively, 
as their original definitions were
\begin{align*}
    \cwd - e^{-2}\kappa\cwd x_{1,1} &= \cwd(1-e^{-2}\kappa x_{1,1}) \qquad\text{and} \\
    \cxwd - e^{-2}\kappa\cxwd x_{1,1} &= \cxwd(1-e^{-2}\kappa x_{1,1}),
\end{align*}
which differ from $\cwd$ and $\cxwd$ by units.
Notice also that $\divp\divq = \cwd\cxwd$ with either definition.

We will explain the existence of this unit further in \cite{CH:QuadricsIII},
where the identification of $\Qpq 22$ with $\Xpq 1{}\times\Xpq 1{}$
will identify this unit with one found in \cite{Co:InfinitePublished}.
\end{remark}

\begin{remark}
The relation for $x_{p,q}^2$ was proved algebraically above
and we repeat the idea here for clarity, then give a geometric justification.
Using the relation $e^2 = \zeta_0\cwd - (1-\kappa)\zeta_1\cxwd$, we can write
$x_{p,q}^2$ in the following form.
\[
    x_{p,q}^2 = \begin{cases}
                    0 & \text{if $p$ and $q$ are even} \\
                    \zeta_0\cwd[p]\cxwd[q-1]x_{p,q} - (1-\kappa)\zeta_1\cwd[p-1]\cxwd[q]x_{p,q} 
                        & \text{if $p$ and $q$ are odd} \\
                    \zeta_0\cwd[p]\cxwd[q-1]x_{p,q} & \text{if $p$ is even and $q$ is odd} \\
                    \zeta_1\cwd[p-1]\cxwd[q]x_{p,q} & \text{if $p$ is odd and $q$ is even.}
                \end{cases} \\
\]
Applying $\rho$, this reduces to the known nonequivariant relation in $H^{4(p+q-1)}(\eQp{2(p+q)})$:
\[
    y^2 = \begin{cases}
            0 & \text{if $p+q$ is even} \\
            c^{p+q-1}y & \text{if $p+q$ is odd}
          \end{cases}
\]
On the other hand, on taking fixed points, we see the nonequivariant relations that hold
in $H^{2(p-1)}(\eQp{2p})\dirsum H^{2(q-1)}(\eQp{2q})$:
\[
    (y^2, y^2) =
    \begin{cases}
        (0,0) & \text{if $p$ and $q$ are even} \\
        (c^{p-1}y, c^{q-1}y) & \text{if $p$ and $q$ are odd} \\
        (0,c^{q-1}y) & \text{if $p$ is even and $q$ is odd} \\
        (c^{p-1}y, 0) & \text{if $p$ is odd and $q$ is even.}
    \end{cases}
\]
(We use the fact that $\rho(1-\kappa) = 1$ but $(1-\kappa)^\GG = -1$.)
The expression for $x_{p,q}^2$ gives the unique element that
respects these nonequivariant relations.

Here is a geometric explanation, based on the idea that $x_{p,q}^2$
should be represented by some $0$-dimensional $\GG$-set in $\Qpq{2p}{2q}$. We defined
\[
    x_{p,q} = [j(\Xpq pq)]^*
\]
where
\[
    j(\Xpq pq) = \{ [\vec 0:\vec u:\vec 0:\vec v] \in \Qexpq{2p}{2q} \}.
\]
Looking at other isotropic affine subspaces and paying attention to the parity of codimensions,
we can also write
\[
    x_{p,q} = [Z]^*
\]
where $Z\subset \Qexpq{2p}{2q}$ is the following subvariety, 
in which we write, for example, $\vec x_{p}$ to indicate
an arbitrary vector of length $p$:
\[
    Z = 
    \begin{cases}
        \{[\vec x_p:\vec 0_p:\vec y_q:\vec 0_q]\} & \text{if $p$ and $q$ are even} \\
        \{[\vec x_{p-1}:0:u:\vec 0_{p-1}:\vec y_{q-1}:0:v:\vec 0_{q-1}] \}
            &\text{if $p$ and $q$ are odd} \\
        \{[\vec x_p:\vec 0_p:\vec y_{q-1}:0:v:\vec 0_{q-1}]\}
            &\text{if $p$ is even and $q$ is odd} \\
        \{[\vec x_{p-1}:0:u:\vec 0_{p-1}:\vec y_q:\vec 0_q]\}
            &\text{if $p$ is odd and $q$ is even.}
    \end{cases}
\]
We can then compute $x_{p,q}^2$ by looking at the intersection
\[
    j(\Xpq pq)\intersect Z =
    \begin{cases}
        \emptyset & \text{if $p$ and $q$ are even} \\
        \{[\vec 0_p:1:\vec 0_{p-1}:\vec 0_q:\vec 0_q], \\
         \quad [\vec 0_p:\vec 0_p:\vec 0_q:1:\vec 0_{q-1}], \\
         \quad [\vec 0_p:1:\vec 0_{p-1}:\vec 0_q:\pm 1:\vec 0_{q-1}] \}
            & \text{if $p$ and $q$ are odd} \\
        [\vec 0_p:\vec 0_p:\vec 0_q:1:\vec 0_{q-1}]
            & \text{if $p$ is even and $q$ is odd} \\
        [\vec 0_p:1:\vec 0_{p-1}:\vec 0_q:\vec 0_q]
            & \text{if $p$ is odd and $q$ is even.}
    \end{cases}
\]
The first case clearly represents 0. The third case is a single point pushed forward from $\Xpq pq$ along $j$:
\[
    [j(\Xq{})]^* = j_!(\zeta_0\cwd[p]\cxwd[q-1]) = \zeta_0\cwd[p]\cxwd[q-1]x_{p,q}
\]
(where the factor of $\zeta_0$ is present to put the element in the correct grading).
Similarly, the fourth case is
\[
    [j(\Xp{})]^* = j_!(\zeta_1\cwd[p-1]\cxwd[q]) = \zeta_1\cwd[p-1]\cxwd[q]x_{p,q}.
\]
The second case is more interesting. The $\GG$-set given by the intersection consists
of two fixed points and one free orbit. The two fixed points represent, as above,
$\zeta_0\cwd[p]\cxwd[q-1]x_{p,q}$ and $\zeta_1\cwd[p-1]\cxwd[q]x_{p,q}$.
The free orbit we can take as representing $g\zeta_1\cwd[p-1]\cxwd[q]x_{p,q}$.
(We could also use $g\zeta_0\cwd[p]\cxwd[q-1]x_{p,q}$; the two elements are equal.)
However, these points may have multiplicities associated with them, which we can resolve
by looking at them nonequivariantly, where the four points must cancel to give 0 because $p+q$ is even,
and at fixed points, where we must end up with a single positive point in each fixed set component
because $p$ and $q$ are odd.
The result is that we must have
\begin{align*}
    x_{p,q}^2 &= \zeta_0\cwd[p]\cxwd[q-1]x_{p,q} + \zeta_1\cwd[p-1]\cxwd[q]x_{p,q}
                - g\zeta_1\cwd[p-1]\cxwd[q]x_{p,q} \\
    &= \zeta_0\cwd[p]\cxwd[q-1]x_{p,q} - (1-\kappa)\zeta_1\cwd[p-1]\cxwd[q]x_{p,q}
\end{align*}
when $p$ and $q$ are odd.
\end{remark}

\begin{remark}
Finally, as for the other types of quadrics, we can compute $\cwd[p]\cxwd[q]$. The result is
\[
    \cwd[p]\cxwd[q] = \bigl(\tau(\iota^{-2})\zeta_0\cwd + e^{-2}\kappa\cwd\cxwd\bigr) x_{p,q}
\]
This relation can be understood as coming from the similar relation in either
$\Qpq{2p-1}{2q}$ or $\Qpq{2p}{2q-1}$
given in
Remark~\ref{rem:db cwd p cxwd q}, pushed forward along the map $k$ or $\ell$, respectively,
as used in Proposition~\ref{prop:dd divisible}.
(For this we use the relation $\tau(\iota^{-2})\zeta_0\cwd = \tau(\iota^{-2})\zeta_1\cxwd$.)
\end{remark}

\bibliography{Bibliography}{}
\bibliographystyle{amsplain} 

\end{document}